\documentclass[final]{amsart}
\usepackage{amssymb,amsmath, color}

\usepackage{amssymb}

\newtheorem{Theorem}{Theorem}[section]
\newtheorem{lemma}{Lemma}[section]

\newtheorem{proposition}{Proposition}[section]

\theoremstyle{remark}
\newtheorem{remark}{Remark}[section]

\newcommand{\argmin}{{\rm argmin}\kern 0.12em}

\usepackage{ulem}

\begin{document}

\title{FAST  INERTIAL DYNAMICS AND FISTA ALGORITHMS IN CONVEX OPTIMIZATION.
PERTURBATION ASPECTS.}

\author{Hedy Attouch}

\address{Institut de Math\'ematiques et Mod\'elisation de Montpellier, UMR 5149 CNRS, Universit\'e Montpellier 2, place Eug\`ene Bataillon,
34095 Montpellier cedex 5, France}
\email{hedy.attouch@univ-montp2.fr}

\author{Zaki Chbani}
\address{Laboratoire IBN AL-BANNA de Math\'ematiques et applications (LIBMA),
Cadi Ayyad university, Faculty of Sciences Semlalia, Mathematics, 40000 Marrakech, Morroco}
\email{chbaniz@uca.ma}

\date{July 6, 2015}



\begin{abstract}
In a Hilbert space setting $\mathcal H$,  we study the  fast convergence properties as $t \to + \infty$ of the trajectories of  the second-order differential equation
\begin{equation*}
 \ddot{x}(t) + \frac{\alpha}{t} \dot{x}(t) + \nabla \Phi (x(t)) = g(t),
\end{equation*}
where $\nabla\Phi$ is the gradient of a convex continuously differentiable function $\Phi: \mathcal H \to  \mathbb R$, $\alpha$ is a positive parameter, and $g:   [t_0, + \infty[ \rightarrow \mathcal H$ is a "small" perturbation term. 
In this damped inertial system, the viscous  damping coefficient $\frac{\alpha}{t}$ vanishes asymptotically, but not too rapidly.
 For $\alpha \geq 3$, and $\int_{t_0}^{+\infty} t \|g(t)\| dt < + \infty$, just assuming that $\argmin \Phi \neq \emptyset$, we show that any  trajectory of the above system satisfies the fast convergence property  
\begin{align*}
  \Phi(x(t))-  \min_{\mathcal H}\Phi \leq \frac{C}{t^2}.
\end{align*}
For $\alpha > 3$,  we show that any trajectory converges weakly to a minimizer of $\Phi$, and we show the strong convergence property in various practical situations.
This complements the results obtained  by Su-Boyd- Cand\`es, and Attouch-Peypouquet-Redont,
in the unperturbed case $g=0$. 
The parallel study of the time discretized version of this system provides  new insight on the effect of errors, or perturbations on   Nesterov's type algorithms. We obtain fast convergence of the values, and convergence of the trajectories  for a perturbed version of the variant of FISTA recently considered by Chambolle-Dossal, and Su-Boyd-Cand\`es.

\end{abstract}

\thanks{With the support of  ECOS  grant C13E03,
Effort sponsored by the Air Force Office of Scientific Research, Air Force Material Command, USAF, under grant number FA9550-14-1-0056.}
\maketitle

\vspace{0.3cm}

\section{Introduction}
Throughout the paper, $\mathcal H$ is a real Hilbert space which is endowed with the scalar product $\langle \cdot,\cdot\rangle$, with $\|x\|^2= \langle x,x\rangle    $ for any $x\in \mathcal H$.
Let $\Phi : \mathcal H \rightarrow \mathbb R$ be a convex differentiable function,  whose gradient   $\nabla \Phi$ is Lipschitz continuous on  bounded sets. We suppose that   $S=\argmin \Phi$ is nonempty. Let us give $\alpha$   a positive parameter.
We are going to study the asymptotic behaviour (as $t \to + \infty$) of the trajectories of the second-order differential equation 
\begin{equation}\label{edo01}
 \mbox{(AVD)}_{\alpha,g} \quad \quad \ddot{x}(t) + \frac{\alpha}{t} \dot{x}(t) + \nabla \Phi (x(t)) = g(t)
\end{equation}
and consider similar questions for the corresponding algorithms.
 Let us give some $t_0 >0$. The second-member $g:   [t_0, + \infty[ \rightarrow \mathcal H$ is  a perturbation term (integrable source term), such that $g(t)$ is small for  large $t$. Precisely, in our
main result,  Theorem \ref{fastconv-thm},  assuming that $\alpha \geq 3$, and $\int_{t_0}^{+\infty} t \|g(t)\| dt < + \infty$, we  show that any  trajectory of (\ref{edo01}) satisfies the fast convergence property  
\begin{align}\label{basic-fast}
  \Phi(x(t))-  \min_{\mathcal H}\Phi \leq \frac{C}{t^2}.
\end{align}
This extends the fast convergence of the values  obtained by Su, Boyd and 
Cand\`es in \cite{SBC} in the unperturbed case $g=0$. 
In Theorem \ref{Thm-weak-conv},  when   $\alpha > 3$, we show that any  trajectory of (\ref{edo01})
converges weakly to a minimizer of $\Phi$, which extends  
 the convergence result obtained by Attouch, Peypouquet, and Redont in \cite{APR1} in the case $g=0$. 
  
This inertial system involves a  viscous damping which is attached to the term $\frac{\alpha}{t} \dot{x}(t)$. It is an isotropic linear damping
with a viscous parameter $\frac{\alpha}{t}$ which vanishes asymptotically, but not too rapidly. The asymptotic behaviour of the inertial gradient-like system 
\begin{equation}\label{edo2}
\mbox{(AVD)} \quad \quad \ddot{x}(t) + a(t) \dot{x}(t)  + \nabla \Phi (x(t)) = 0,
\end{equation}
with Asymptotic Vanishing Damping ((AVD) for short), has been studied by Cabot, Engler  and Gaddat in \cite{CEG1}-\cite{CEG2}. As a main result, they proved that, under moderate decrease of $a(\cdot)$ to zero, 
i.e., $a(t) \to 0$ as $t \to +\infty$ with $\int_0^{\infty} a(t) dt = + \infty$, then for any trajectory $x(\cdot)$ of (\ref{edo2})
\begin{equation}\label{edo3}
 \Phi (x(t)) \to \min_{\mathcal H}\Phi.
\end{equation}  
 As a striking property, for the specific choice $a(t)= \frac{\alpha}{t}$, with $\alpha \geq 3$ , for example when considering 
\begin{equation}\label{edo4}
\ddot{x}(t) + \frac{3}{t} \dot{x}(t)  + \nabla \Phi (x(t)) = 0,
\end{equation}
 it has been proved by  Su, Boyd, and  Cand\`es in \cite{SBC}  that  the fast convergence property  of the values (\ref{basic-fast}) is satisfied by the trajectories of 
 (\ref{edo4}).
In the same article \cite{SBC}, the authors show that (\ref{edo4}) can be seen as a continuous version of the  fast convergent method of Nesterov, see \cite{Nest1}-\cite{Nest2}-\cite{Nest3}-\cite{Nest4}.
For the continuous dynamic, a related study concerning the case $a(t)= \frac{1}{t^{\theta}}$, $0<\theta <1$ has been developed by Jendoubi and May in \cite{JM}, with roughly speaking $\mathcal O (\frac{1}{t^{1+ \theta}})$ convergence. The analysis developped in \cite{JM} does not contain the case $a(t)= \frac{\alpha}{t}$, where the introduction of an additional scaling, due to the coefficient $\alpha$, requires a specific analysis. That's our main concern in this paper.

Our results provide new insight on the effect of perturbations or errors in the associated algorithms. They provide a guideline for the study of the preservation, under small perturbations, of the fast convergence property of the corresponding Nesterov type algorithms.
Specifically we consider a perturbed version of the variant of FISTA recentely considered by Chambolle and Dossal \cite{CD},  and Su, Boyd and 
Cand\`es  \cite{SBC}.
We obtain fast convergence of the values in the case $\alpha \geq 3$, and convergence of the trajectories in the case $\alpha > 3$.
Convergence of the trajectories in the case $\alpha =3$,  which corresponds to Nesterov algorithm, is still an open question.
\section{Fast Convergence of the values}
Let $\Phi : \mathcal H \rightarrow \mathbb R$ be a convex function,  whose gradient   $\nabla \Phi$ is Lipschitz continuous on  bounded sets.
Let $t_0 >0$, $\alpha >0$, and  $g:   [t_0, + \infty[ \rightarrow \mathcal H$ such that $\displaystyle{\int_{t_0}^{+\infty} \|g(t)\| dt < + \infty}$.
We consider the second-order differential equation
\begin{equation}\label{energy-00}
 \mbox{{\rm(AVD)}}_{\alpha, g}  \quad \ddot{x}(t) + \frac{\alpha}{t} \dot{x}(t) + \nabla \Phi (x(t)) = g(t).\quad\quad
\end{equation}
From Cauchy-Lipschitz theorem, for any Cauchy data $x(t_0) =x_0 \in \mathcal H, \  \dot{x}(t_0)= x_1 \in \mathcal H    $ we immediately infer the existence and uniqueness of a local solution to (\ref{energy-00}).
The global existence follows from the  energy estimate proved in Proposition \ref{energy-thm-1}, in  the next paragraph.
Throughout this paper we will use the following Gronwall-Bellman lemma, see  \cite[Lemme A.5]{Bre1} for a proof.

\begin{lemma}\label{GB-lemma}
Let $m\in L^1 (t_0,T; \mathbb R)$ such that $m \geq 0$ a.e. on $]t_0,T[$ and let \ $c$ \ be a nonnegative constant.
Suppose that $w$ is a continuous function from $[t_0,T]$
into $\mathbb R$ that satisfies, for all $t\in[t_0,T]$
$$
\frac{1}{2}w^2 (t) \leq 
\frac{1}{2} c^2 + \int_{t_0}^t m(\tau) w(\tau) d\tau .
$$
Then, for all $t\in[t_0,T]$
$$
|w(t| \leq c + \int_{t_0}^t m(\tau)  d\tau .
$$
\end{lemma}

\subsection{Energy estimates}

The following estimates are obtained by considering the global energy of the system, and showing that it is a strict Lyapunov function.
\begin{proposition}\label{energy-thm-1}
Suppose $\alpha >0 $, and $\displaystyle{\int_{t_0}^{+\infty} \|g(t)\| dt < + \infty}$ .   Then, for any  orbit  $x: [t_0, +\infty[ \rightarrow \mathcal H$   of $ \mbox{{\rm(AVD)}}_{\alpha,g}$
\begin{align}
  &  \sup_t \|  \dot{x}(t)  \| < + \infty  ,   \label{et1}    \\ 
  & \int_{t_0}^{+\infty} \frac{1}{t} \| \dot{x}(t)  \|^2 dt < + \infty . \label{et2}   
\end{align}
Precisely, for any $t\geq t_0$
\begin{equation}\label{energy-001}
\| \dot{x}(t) \| \leq \| \dot{x}(t_0) \| + \sqrt{2}\left(  \Phi(x_0)-  \min_{\mathcal H}\Phi   \right) + \int_{t_0}^{\infty} \| g(\tau) \| d\tau ,
\end{equation}
\begin{equation}\label{energy-002}
\int_{t_0}^{\infty} \frac{\alpha}{\tau}  \| \dot{x}(\tau)  \|^2 d\tau \leq \frac{1}{2} \| \dot{x}(t_0) \|^2   + \left(  \Phi(x_0)-  \min_{\mathcal H}\Phi   \right) + 
\left( \| \dot{x}(t_0) \| + \sqrt{2}\left(  \Phi(x_0)-  \min_{\mathcal H}\Phi   \right) + \|g\|_{L^1 (t_0, \infty)} \right)  \|g\|_{L^1 (t_0, \infty)}  .
\end{equation}

\end{proposition}
\begin{proof} 
Let us give some $T >t_0$. For $t_0 \leq t \leq T$, let us define the energy function
\begin{equation}\label{energy-01}
W_T(t):= \frac{1}{2} \| \dot{x}(t) \|^2   + \left(  \Phi(x(t))-  \min_{\mathcal H}\Phi   \right)  +  
\int_t^T \langle   \dot{x}(\tau) , 
g(\tau) \rangle d\tau .
\end{equation}
Because of $\dot{x}$ continuous, and $g$ integrable, the energy function $W_T$ is well defined. After time derivation of $W_T$, and by using $ \mbox{{\rm(AVD)}}_{\alpha,g}$, we obtain
\begin{align*}
\dot{W_T} (t)&:=  \langle   \dot{x}(t),    \ddot{x}(t) +  \nabla \Phi (x(t)) - g(t) \rangle \\
&=   \langle  \dot{x}(t), -\frac{\alpha}{t}  \dot{x}(t)       \rangle ,
\end{align*}
that is
\begin{align}\label{energy-02}
\dot{W_T} (t) + \frac{\alpha}{t}  \| \dot{x}(t)  \|^2 \leq 0.
\end{align}
Hence $W_T(\cdot)$ is a decreasing function.
In particular, $ W_T(t) \leq W_T(t_0)$, i.e.,
$$
\frac{1}{2} \| \dot{x}(t) \|^2   + \left(  \Phi(x(t))-  \min_{\mathcal H}\Phi   \right)  +  
\int_t^T \langle   \dot{x}(\tau) , 
g(\tau) \rangle d\tau
 \leq \frac{1}{2} \| \dot{x}(t_0) \|^2   + \left(  \Phi(x_0)-  \min_{\mathcal H}\Phi   \right)  +  
\int_{t_0}^T \langle   \dot{x}(\tau) , 
g(\tau) \rangle d\tau . 
$$
As a consequence, 
$$
\frac{1}{2} \| \dot{x}(t) \|^2  \leq \frac{1}{2} \| \dot{x}(t_0) \|^2   + \left(  \Phi(x_0)-  \min_{\mathcal H}\Phi   \right)  + \int_{t_0}^t \|\dot{x}(\tau)\| \| g(\tau) \| d\tau.
$$
Applying Gronwall-Bellman lemma \ref{GB-lemma}, we obtain 
\begin{align*}
\| \dot{x}(t) \| &\leq \left(  \| \dot{x}(t_0) \|^2   + 2\left(  \Phi(x_0)-  \min_{\mathcal H}\Phi   \right) \right)^{\frac{1}{2}}  + \int_{t_0}^t \| g(\tau) \| d\tau \\
& \leq \| \dot{x}(t_0) \| + \sqrt{2}\left(  \Phi(x_0)-  \min_{\mathcal H}\Phi   \right) + \int_{t_0}^t \| g(\tau) \| d\tau .
\end{align*}
This being true for arbitrary $T>t_0$, and $t_0 \leq t \leq T$, we deduce that
\begin{equation}\label{energy-03}
\| \dot{x}(t) \| \leq \| \dot{x}(t_0) \| + \sqrt{2}\left(  \Phi(x_0)-  \min_{\mathcal H}\Phi   \right) + \int_{t_0}^{\infty} \| g(\tau) \| d\tau ,
\end{equation}
which gives (\ref{et1}) and (\ref{energy-001}).
As a consequence, the function $W$ (corresponding to 
$T= +\infty$) 
\begin{equation}\label{energy-04}
W(t):= \frac{1}{2} \| \dot{x}(t) \|^2   + \left(  \Phi(x(t))-  \min_{\mathcal H}\Phi   \right)  +  
\int_t^{\infty} \langle   \dot{x}(\tau) , 
g(\tau) \rangle d\tau ,
\end{equation}
is well defined, 
and is minorized by 
\begin{equation}\label{energy-05}
-\|  \dot{x} \|_{ L^{\infty} (t_0, + \infty)}  \int_{t_0}^{\infty} \| g(\tau) \| d\tau .
\end{equation}
By (\ref{energy-02}) we have
\begin{align}\label{energy-06}
\dot{W} (t) + \frac{\alpha}{t}  \| \dot{x}(t)  \|^2 \leq 0.
\end{align}
 Integrating (\ref{energy-06}) from $t_0$ to $t$, and using (\ref{energy-03}),   (\ref{energy-05}), we obtain
\begin{align*}
\int_{t_@}^{\infty} \frac{\alpha}{\tau}  \| \dot{x}(\tau)  \|^2 d\tau  & \leq \frac{1}{2} \| \dot{x}(t_0) \|^2   + \left(  \Phi(x(t_0))-  \min_{\mathcal H}\Phi   \right) + \|  \dot{x} \|_{ L^{\infty} (t_0, + \infty)}  \int_{t_0}^{\infty} \| g(\tau) \| d\tau < + \infty\\
& \leq \frac{1}{2} \| \dot{x}(t_0) \|^2   + \left(  \Phi(x(t_0))-  \min_{\mathcal H}\Phi   \right) + 
\left( \| \dot{x}(t_0) \| + \sqrt{2}\left(  \Phi(x_0)-  \min_{\mathcal H}\Phi   \right) + \|g\|_{L^1 (t_0, +\infty)} \right)  \|g\|_{L^1 (t_0, +\infty)} ,
\end{align*}
which gives (\ref{et2}) and (\ref{energy-002}).
\end{proof}

\subsection{The main result}

Let us state our main result.

\begin{Theorem}\label{fastconv-thm}
Suppose that $\alpha \geq3$, and $\displaystyle{\int_{t_0}^{+\infty} \tau \|g(\tau)\| d\tau < + \infty}$.   Then, for any orbit  $x: [t_0, +\infty[ \rightarrow \mathcal H$   of $ \mbox{{\rm(AVD)}}_{\alpha,g}$, we have the following fast convergence of the values: 
 $$ \Phi(x(t))-  \min_{\mathcal H}\Phi = \mathcal O  \left( \frac{1}{t^2}\right) .$$
Precisely 
\begin{align}\label{Liap-001}
 \frac{2}{\alpha-1}t^2 (\Phi(x(t))- \inf_{\mathcal H}\Phi ) \leq C+ 
 2 \left(\left( \frac{C}{\alpha-1}\right)^{ \frac{1}{2} } + \frac{1}{\alpha-1}
 \int_{t_{0}}^{\infty}  \tau \| g(\tau)   \| d\tau \right) \int_{t_{0}}^{\infty}  \tau \|g(\tau) \|   d\tau ,
\end{align}
with
$$
C= \frac{2}{\alpha-1}t^2 (\Phi(x_0)- \inf_{\mathcal H}\Phi )+ (\alpha-1) \| x_0 - x^{*} + \frac{t_0}{\alpha-1}  \dot{x}(t_0)  \|^2 .
$$
Moreover
\begin{align}\label{Liap-001-b}
\sup_{t \geq t_0}\| x(t) - x^{*} + \frac{t}{\alpha-1}  \dot{x}(t)  \| \leq \left( \frac{C}{\alpha-1}\right)^{ \frac{1}{2} } + \frac{1}{\alpha-1}
 \int_{t_{0}}^{\infty}  \tau \| g(\tau)   \| d\tau <+\infty .
\end{align}
\end{Theorem}
\begin{proof} The proof is an adaptation to our setting 
(with an integrable source term $g$) of the argument developed by Su-Boyd-Cand\`es in \cite{SBC}. 
Let us give some $T >t_0$, and $x^{*} \in S= \mbox{\rm{argmin}} \Phi$. For $t_0 \leq t \leq T$, let us define the energy function
\begin{equation}\label{basic-Liap}
 \mathcal E_{\alpha,g, T} (t):= \frac{2}{\alpha-1}t^2 (\Phi(x(t))- \inf_{\mathcal H}\Phi )+ (\alpha-1) \| x(t) - x^{*} + \frac{t}{\alpha-1}  \dot{x}(t)  \|^2 + 
 2\int_t^{T} \tau \langle x(\tau) - x^{*} + \frac{\tau}{\alpha-1}  \dot{x}(\tau),  g(\tau)   \rangle d\tau.
\end{equation}
Let us show that
$$ \dot{\mathcal E}_{\alpha,g,T} (t) +  2 \displaystyle{\frac{\alpha-3}{\alpha-1}} t (\Phi(x(t))- \min_{\mathcal H}\Phi ) \leq 0 .
$$
Derivation of $\mathcal E_{\alpha, g,T} (\cdot)$ gives
\begin{align*}
\dot{\mathcal E}_{\alpha,g,T} (t)&:= \frac{4}{\alpha-1}t (\Phi(x(t))- \inf_{\mathcal H}\Phi ) +
\frac{2}{\alpha-1}t^2  \langle  \nabla \Phi (x(t)),  \dot{x}(t)       \rangle \\
&+ 2(\alpha-1) \langle  x(t) - x^{*} + \frac{t}{\alpha-1} \dot{x}(t)  ,  \dot{x}(t)       +
\frac{1}{\alpha-1} \dot{x}(t)   +\frac{t}{\alpha-1}  \ddot{x}(t)
 \rangle - 2t\langle  x(t) - x^{*} + \frac{t}{\alpha-1} \dot{x}(t)  ,  g(t) \rangle \\
&=  \frac{4}{\alpha-1}t (\Phi(x(t))- \inf_{\mathcal H}\Phi ) + 
\frac{2}{\alpha-1}t^2  \langle  \nabla \Phi (x(t)),  \dot{x}(t)       \rangle \\
&+  2(\alpha-1) \langle  x(t) - x^{*} + \frac{t}{\alpha-1}  \dot{x}(t) , 
\frac{t}{\alpha-1} \Big( \frac{\alpha}{t}\dot{x}(t) + \ddot{x}(t) -g(t)\Big) \rangle. 
\end{align*}
Then use $\mbox{{\rm(AVD)}}_{\alpha, g}$ in this last expression to obtain
\begin{align}
\dot{\mathcal E}_{\alpha,g,T} (t)=& \frac{4}{\alpha-1}t (\Phi(x(t))- \inf_{\mathcal H}\Phi ) + \frac{2}{\alpha-1}t^2  \langle  \nabla \Phi (x(t)),  \dot{x}(t)  \rangle \\
&-  2 t\langle  x(t) - x^{*} +  \frac{t}{\alpha-1}  \dot{x}(t), \nabla \Phi (x(t))   \rangle\\
=& \frac{4}{\alpha-1}t (\Phi(x(t))- \inf_{\mathcal H}\Phi ) -  2 t \langle  x(t) - x^{*} , \nabla \Phi (x(t))   \rangle.\label{rfastoche1}
\end{align}
By convexity of $\Phi$ 
$$
\Phi (x^{*}) \geq \Phi (x(t))+ \langle  x^{*} - x(t)  , \nabla \Phi (x(t)) \rangle. 
$$
Replacing in (\ref{rfastoche1}) we obtain 
\begin{equation*}
\dot{\mathcal E}_{\alpha,g,T} (t) + \left(2- \frac{4}{\alpha-1}\right) t (\Phi(x(t))- \inf_{\mathcal H}\Phi)     \leq 0.
\end{equation*}
Equivalently
\begin{equation}\label{basic-energy-Liap}
 \dot{\mathcal E}_{\alpha,g,T} (t) +  2 \frac{\alpha-3}{\alpha-1} t (\Phi(x(t))- \inf_{\mathcal H}\Phi ) \leq 0 .
\end{equation}
As a consequence, for $\alpha \geq 3$, the function ${\mathcal E}_{\alpha,g}$ is nonincreasing. In particular,
${\mathcal E}_{\alpha,g} (t) \leq {\mathcal E}_{\alpha,g} (t_0)$, which gives
\begin{align*}
&\frac{2}{\alpha-1}t^2 (\Phi(x(t))- \inf_{\mathcal H}\Phi )+ (\alpha-1) \| x(t) - x^{*} + \frac{t}{\alpha-1}  \dot{x}(t)  \|^2 + 
 2\int_t^{T} \tau \langle x(\tau) - x^{*} + \frac{\tau}{\alpha-1}  \dot{x}(\tau),  g(\tau)   \rangle d\tau \\
&\leq  \frac{2}{\alpha-1}{t_0}^2 (\Phi(x_0)- \inf_{\mathcal H}\Phi )+ (\alpha-1) \| x_0 - x^{*} + \frac{t_0}{\alpha-1}  \dot{x}(t_0)  \|^2 + 
 2\int_{t_{0}}^{T} \tau \langle x(\tau) - x^{*} + \frac{\tau}{\alpha-1}  \dot{x}(\tau),  g(\tau)   \rangle d\tau .
\end{align*}
Equivalently
\begin{align}\label{energy-08}
\frac{2}{\alpha-1}t^2 (\Phi(x(t))- \inf_{\mathcal H}\Phi )+ (\alpha-1) \| x(t) - x^{*} + \frac{t}{\alpha-1}  \dot{x}(t)  \|^2  \leq C+
 2\int_{t_{0}}^{t} \tau \langle x(\tau) - x^{*} + \frac{\tau}{\alpha-1}  \dot{x}(\tau),  g(\tau)   \rangle d\tau ,
\end{align}
with
$$
C= \frac{2}{\alpha-1}{t_0}^2 (\Phi(x_0)- \inf_{\mathcal H}\Phi )+ (\alpha-1) \| x_0 - x^{*} + \frac{t_0}{\alpha-1}  \dot{x}(t_0)  \|^2 .
$$
From (\ref{energy-08}) we infer
\begin{align}\label{energy-09}
 \frac{1}{2} \| x(t) - x^{*} + \frac{t}{\alpha-1}  \dot{x}(t)  \|^2  \leq \frac{C}{2(\alpha-1)}+ \frac{1}{\alpha-1}
 \int_{t_{0}}^{t}  \| x(\tau) - x^{*} + \frac{\tau}{\alpha-1}  \dot{x}(\tau)\|  \|\tau g(\tau)   \| d\tau .
\end{align}
Applying once more  Gronwall-Bellman lemma \ref{GB-lemma}, we obtain
\begin{align}\label{energy-1}
 \| x(t) - x^{*} + \frac{t}{\alpha-1}  \dot{x}(t)  \| \leq \left( \frac{C}{\alpha-1}\right)^{ \frac{1}{2} } + \frac{1}{\alpha-1}
 \int_{t_{0}}^{t}  \tau \| g(\tau)   \| d\tau .
\end{align}
Since $\displaystyle{\int_{t_0}^{+\infty} t \|g(t)\| dt < + \infty}$, it follows that
\begin{align}\label{energy-10}
\sup_t \| x(t) - x^{*} + \frac{t}{\alpha-1}  \dot{x}(t)  \| \leq \left( \frac{C}{\alpha-1}\right)^{ \frac{1}{2} } + \frac{1}{\alpha-1}
 \int_{t_{0}}^{\infty}  \tau \| g(\tau)   \| d\tau <+\infty .
\end{align}
Returning to (\ref{energy-08}), we conclude  that 
\begin{align}\label{energy-11}
 \frac{2}{\alpha-1}t^2 (\Phi(x(t))- \inf_{\mathcal H}\Phi ) \leq C+ 
 2 \left(\left( \frac{C}{\alpha-1}\right)^{ \frac{1}{2} } + \frac{1}{\alpha-1}
 \int_{t_{0}}^{\infty}  \tau \| g(\tau)   \| d\tau \right) \int_{t_{0}}^{\infty}  \tau \|g(\tau) \|   d\tau .
\end{align}
\end{proof}

\begin{remark} As a consequence 
the energy function
\begin{equation}\label{basic-Liap-c}
 \mathcal E_{\alpha,g} (t):= \frac{2}{\alpha-1}t^2 (\Phi(x(t))- \inf_{\mathcal H}\Phi )+ (\alpha-1) \| x(t) - x^{*} + \frac{t}{\alpha-1}  \dot{x}(t)  \|^2 + 
 2\int_t^{+\infty} \tau \langle x(\tau) - x^{*} + \frac{\tau}{\alpha-1}  \dot{x}(\tau),  g(\tau)   \rangle d\tau.
\end{equation}
is well defined, and is a Lyapunov function for the dynamical system $ \mbox{{\rm(AVD)}}_{\alpha,g}$.
\end{remark}

\section{Convergence of trajectories}

 In the case $\alpha >3$, provided that the second member $g(t)$ is sufficiently small for large $t$, we are going to show the convergence of the trajectories of the  system

\begin{equation*}
 \mbox{(AVD)}_{\alpha,g} \quad \quad \ddot{x}(t) + \frac{\alpha}{t} \dot{x}(t) + \nabla \Phi (x(t)) = g(t).
\end{equation*}

\subsection{Main statement, and  preliminary results}

The following convergence result is an extension to the perturbed case (with a source term $g$) of the convergence result obtained by Attouch-Peypouquet-Redont in \cite{APR1}.

\begin{Theorem} \label{Thm-weak-conv}
Let $\Phi : \mathcal H \rightarrow \mathbb R$ a convex continuously differentiable function 
such that   $S=\argmin \Phi$ is nonempty.
Suppose that $\alpha >3$ and $\displaystyle{\int_{t_0}^{+\infty} t \|g(t)\| dt < + \infty}$. 
Let $t_0 >0$, and  $x: [t_0, +\infty[ \rightarrow \mathcal H$ be a  classical solution  of {\rm$  \mbox{(AVD)}_{\alpha,g}$}. 
Then, the following  convergence properties hold: 

\medskip

a) (weak convergence) \ 
There exists some  $x^{*}  \in \argmin \Phi$ such that
\begin{equation}\label{conv-basic}
  x(t) \rightharpoonup x^{*} \ \mbox{weakly as} \ t \rightarrow + \infty.
\end{equation}

b) (fast convergence) \ There exists a positive constant $C$ such that 
\begin{equation}\label{rfast1}
   \Phi(x(t))-  \min_{\mathcal H}\Phi \leq \frac{C}{t^2}
\end{equation}
\begin{equation}\label{energy2}
  \int_{t_0}^{\infty}  t \left( \Phi(x(t)) - \inf_{\mathcal H}\Phi \right) dt  < + \infty.
\end{equation}
c) (energy estimate)
\begin{align}\label{energy1}
  &\int_{t_0}^{\infty}  t\| \dot{x}(t) \|^2  dt   < + \infty\\
  &\| \dot{x}(t)  \| \leq \frac{C}{t} \label{conv1}
\end{align}
and hence
\begin{equation}\label{conv2}
   \lim_{t\to\infty}  \| \dot{x}(t)  \| = 0.
\end{equation}
\end{Theorem}
 In order to analyze the  convergence properties of the trajectories
of system (\ref{edo01}), we will use the  Opial's lemma
\cite{Op} that we recall in its continuous form; see also \cite{Bruck}, who initiated the use of this argument to analyze the asymptotic convergence of nonlinear contraction semigroups in Hilbert spaces.
\begin{lemma}\label{Opial} Let $S$ be a non empty subset of $\mathcal H$
and $x:[0,+\infty[\to \mathcal H$ a map. Assume that 
\begin{eqnarray*}
(i) &  & \mbox{for every }z\in S,\>\lim_{t\to+\infty}\|x(t)-z\|\mbox{ exists};\\
(ii) &  & \mbox{every weak sequential cluster point of the map }x\mbox{ belongs to }S.
\end{eqnarray*}
 Then 
\[
w-\lim_{t\to+\infty}x(t)=x_{\infty}\ \ \mbox{ exists, for some element }x_{\infty}\in S.
\]
 \end{lemma}

We also need the following result concerning the integration of a first-order nonautonomous differential inequation, see \cite{APR1}.

 \begin{lemma}\label{basic-edo} Suppose that $\delta >0$, and let $w: [\delta, +\infty[ \rightarrow \mathbb R$ be a continuously differentiable function that satisfies the following differential inequality
 \begin{equation}\label{basic-edo1}
 \dot{w}(t) + \frac{\alpha}{t} w(t)  \leq k(t),
\end{equation}
 for some $\alpha > 1$, and some nonnegative function $k: [\delta, +\infty[ \to \mathbb R$ such that $t \mapsto tk(t) \in L^1 (\delta, +\infty)$. Then 
 \begin{equation}\label{basic-edo2}
  w^{+} \in L^1 (\delta, +\infty).
\end{equation}
 \end{lemma}
\subsection{Proof of the convergence results}

\begin{proof}

\textbf{Step 1.} 
Let us return to the decrease property (\ref{basic-energy-Liap}) which is satisfied by the Lyapunov function 
${\mathcal E}_{\alpha,g}$:
\begin{equation*}
\dot{\mathcal E}_{\alpha,g} (t) + 2 \frac{\alpha-3}{\alpha-1}  t (\Phi(x(t))- \inf_{\mathcal H}\Phi)  \leq 0.
\end{equation*}
By integration of this inequality, we obtain
\begin{equation*}
{\mathcal E}_{\alpha,g} (t)  +
2 \frac{\alpha-3}{\alpha-1} \int_{t_0}^t \tau(\Phi(x(\tau))- \inf_{\mathcal H}\Phi) d\tau \leq {\mathcal E}_{\alpha,g} (t_0).
\end{equation*}
By definition of ${\mathcal E}_{\alpha,g}$, and neglecting its nonnegative terms, we infer
\begin{equation*} 
 2\int_t^{+\infty} \tau \langle x(\tau) - x^{*} + \frac{\tau}{\alpha-1}  \dot{x}(\tau),  g(\tau)   \rangle d\tau  +
2 \frac{\alpha-3}{\alpha-1} \int_{t_0}^t \tau(\Phi(x(\tau))- \inf_{\mathcal H}\Phi) d\tau \leq {\mathcal E}_{\alpha,g} (t_0).
\end{equation*}
Hence
\begin{equation*} 
2 \frac{\alpha-3}{\alpha-1} \int_{t_0}^t \tau(\Phi(x(\tau))- \inf_{\mathcal H}\Phi) d\tau \leq {\mathcal E}_{\alpha,g} (t_0) + 2\int_{t_0}^{+\infty} \| x(\tau) - x^{*} + \frac{\tau}{\alpha-1}  \dot{x}(\tau)\| \|\tau g(\tau)  \| \rangle d\tau  .
\end{equation*}
By (\ref{energy-10}), we have
\begin{align*}
\sup_t \| x(t) - x^{*} + \frac{t}{\alpha-1}  \dot{x}(t)  \|  <+\infty .
\end{align*}
As a consequence
\begin{equation*} 
2 \frac{\alpha-3}{\alpha-1} \int_{t_0}^t \tau(\Phi(x(\tau))- \inf_{\mathcal H}\Phi) d\tau \leq {\mathcal E}_{\alpha,g} (t_0) + 2 \sup_t \| x(t) - x^{*} + \frac{t}{\alpha-1}  \dot{x}(t)  \|\int_{t_0}^{+\infty}  \|\tau g(\tau)  \| \rangle d\tau  .
\end{equation*}
Since $\alpha >3$, we deduce that
\begin{equation} \label{basic-estim-2}
 \int_{t_0}^{+\infty} \tau(\Phi(x(\tau))- \inf_{\mathcal H}\Phi) d\tau < + \infty  .
\end{equation}

\medskip

\textbf{Step 2.}
Let us show that
\begin{equation*}
  \int_{t_0}^{\infty} t \|\dot{x}(t)  \|^2 dt < +\infty.
\end{equation*}
To that end, we use the energy estimate which is obtained by taking the scalar product of (\ref{edo01}) by $t^2 \dot{x}(t)$:
\begin{equation}\label{scale-energy1}
t^2\langle  \ddot{x}(t), \dot{x}(t) \rangle + 
\alpha t \|\dot{x}(t)\|^2  + t^2 \langle  \nabla \Phi (x(t)), \dot{x}(t) \rangle  = t^2 \langle  g(t), \dot{x}(t) \rangle .
\end{equation}
By the classical derivation chain rule, and Cauchy-Schwarz inequality, we obtain
\begin{equation}\label{scale-energy2}
\frac{1}{2}t^2  \frac{d}{dt}\|\dot{x}(t)\|^2 + 
\alpha t \|\dot{x}(t)\|^2 +  t^2 \frac{d}{dt}\Phi (x(t) \leq   \| t g(t)\| \|t \dot{x}(t)\| .
\end{equation}
After integration by parts
\begin{align*}
&\frac{t^2}{2}\|\dot{x}(t)\|^2 - \frac{{t_0}^2}{2}\|\dot{x}(t_0)\|^2- \int_{t_0}^t  s \|\dot{x}(s)\|^2 ds +  \alpha \int_{t_0}^t s \|\dot{x}(s)\|^2 ds \\
&+ t^2 (\Phi(x(t))- \inf_{\mathcal H}\Phi )
- {t_0}^2 (\Phi(x(t_0))- \inf_{\mathcal H}\Phi )
- 2\int_{t_0}^t s (\Phi(x(s))- \inf_{\mathcal H}\Phi )ds \leq \int_{t_0}^t \| s g(s)\| \|s \dot{x}(s)\| ds.
\end{align*}
As a consequence, for some constant $C\geq 0$, depending only on the Cauchy data,
\begin{equation}\label{scale-energy4}
\frac{t^2}{2}\|\dot{x}(t)\|^2 + (\alpha -1) \int_{t_0}^t  s \|\dot{x}(s)\|^2 ds \leq C+
2\int_{t_0}^t s (\Phi(x(s))- \inf_{\mathcal H}\Phi )ds  + \int_{t_0}^t \| s g(s)\| \|s \dot{x}(s)\| ds.
\end{equation}
By (\ref{basic-estim-2}) we have $\int_{t_0}^{\infty} s (\Phi(x(s))- \inf_{\mathcal H}\Phi )ds <+ \infty$. Moreover $\alpha >1$. As a consequence, from (\ref{scale-energy4}) we deduce that, for some other constant $C$ 
\begin{equation}\label{scale-energy5}
 \frac{1}{2}\|t\dot{x}(t)\|^2  \leq C+
 \int_{t_0}^t \| s g(s)\| \|s \dot{x}(s)\| ds.
\end{equation}
Applying Gronwall-Bellman lemma \ref{GB-lemma}, we obtain
\begin{equation*}
 \|t\dot{x}(t)\|  \leq \sqrt{2C}+
 \int_{t_0}^t \| s g(s)\| ds.
\end{equation*}
Since $\displaystyle{\int_{t_0}^{+\infty} t \|g(t)\| dt < + \infty}$, we infer
\begin{equation}\label{scale-energy7}
\sup_t \|t\dot{x}(t)\|  < + \infty.
\end{equation}
Returning to (\ref{scale-energy4}), we deduce that
\begin{equation}\label{scale-energy8}
 (\alpha -1) \int_{t_0}^t  s \|\dot{x}(s)\|^2 ds \leq C+
2\int_{t_0}^{\infty} s (\Phi(x(s))- \inf_{\mathcal H}\Phi )ds  + \sup_t \|t \dot{x}(t)\|  \int_{t_0}^{\infty} \| s g(s)\| ds,
\end{equation}
which gives
$$
\int_{t_0}^{\infty}  t\| \dot{x}(t) \|^2  dt   < + \infty .
$$
Moreover, combining 
 (\ref{energy-10}),
\begin{align*}
\sup_t \| x(t) - x^{*} + \frac{t}{\alpha-1}  \dot{x}(t)  \|  <+\infty ,
\end{align*}
with (\ref{scale-energy7}), we deduce that
\begin{equation}\label{scale-energy9}
\sup_t \|x(t)\|  < + \infty,
\end{equation}
i.e., all the orbits are bounded.

\medskip

\textbf{Step 3.} 
Our proof of the weak convergence property of the orbits of {\rm$  \mbox{(AVD)}_{\alpha,g}$}
relies on Opial's lemma.
Given $x^{*}  \in \argmin \Phi$, let us define $h: [0, +\infty[ \rightarrow \mathbb R^+$ by
\begin{equation}\label{wconv-01}
 h(t) = \frac{1}{2}\| x(t) - x^{*}\|^2 .
\end{equation}
By the classical derivation chain rule 
\begin{align}\label{wconv20}
 &  \dot{h}(t) =  \langle x(t) - x^{*} , \dot{x}(t)  \rangle,\\
 &   \ddot{h}(t) = \langle x(t) - x^{*} , \ddot{x}(t)  \rangle + \| \dot{x}(t) \|^2 .
\end{align}
Combining these two equations, and  using (\ref{edo01}) we obtain
\begin{align}\label{wconv30}
 \ddot{h}(t) + \frac{\alpha}{t} \dot{h}(t) &=  \| \dot{x}(t) \|^2 + \langle x(t) - x^{*} , \ddot{x}(t) + \frac{\alpha}{t} \dot{x}(t)  \rangle,\\
 & =  \| \dot{x}(t) \|^2 +  \langle x(t) - x^{*} , -\nabla \Phi (x(t)) + g(t)\rangle .
\end{align}
By monotonicity of $\nabla \Phi$ and $\nabla \Phi(x^{*}) = 0 $
\begin{equation}\label{wconv40}
 \langle x(t) - x^{*} , -\nabla \Phi (x(t))  \rangle \leq 0.
\end{equation}
By (\ref{wconv30}) and  (\ref{wconv40}) we infer
\begin{equation}\label{wconv50}
 \ddot{h}(t) + \frac{\alpha}{t} \dot{h}(t) \leq \| \dot{x}(t) \|^2  +       \| x(t) - x^{*} \|  \| g(t) \|.
\end{equation}
Equivalently
\begin{equation}\label{wconv60}
 \ddot{h}(t) + \frac{\alpha}{t} \dot{h}(t) \leq k(t),
\end{equation}
with
$$
k(t):= \| \dot{x}(t) \|^2  +       \| x(t) - x^{*} \|  \| g(t) \|.
$$
By (\ref{scale-energy9}) the orbit is bounded. Hence, for some constant $C\geq 0$
$$
k(t)\leq \| \dot{x}(t) \|^2  + C \| g(t) \|.
$$
By assumption $\int_{t_0}^{+\infty} t \|g(t)\| dt < + \infty$, and by (\ref{energy1})
$\int_{t_0}^{\infty}  t\| \dot{x}(t) \|^2  dt   < + \infty$. Hence $t \mapsto tk(t) \in L^1 (t_0, +\infty)$.
Applying  Lemma \ref{basic-edo}, with $w(t)= \dot{h}(t)$, we deduce that $w^{+} \in L^1 (t_0, +\infty)$.
Equivalently $\dot{h}^+(t) \in L^1 (t_0, +\infty)$, which implies that  
 the limit of $h(t)$ exists, as $t \to + \infty$. This proves item $i)$ of the Opial's lemma.
We complete the proof by observing that  item $ii)$  is  satisfied too. Indeed, since $\Phi (x(t))$ converges to  $\inf \Phi$, we have that every weak sequential cluster point of $x(\cdot)$ is a minimizer of $\Phi$.
\end{proof}
\subsection{Strong convergence results}
Since the work of J.B. Baillon, we know that without additional assumptions, 
the trajectories of the gradient systems may not converge strongly.
Let's examine some practical interest situations where 
strong convergence of the trajectories of $ \mbox{{\rm(AVD)}}_{\alpha,g}$ is satisfied.

\smallskip

\noindent \textbf{Strong convergence under $int(\argmin \Phi)\neq \emptyset $.}\\
We will need the following result, see (\cite{APR1}, Lemma 5.4).
\begin{lemma}\label{strong-lem1} 
Suppose that $\delta > 0$, and let $f : [\delta ;+\infty [ \to \mathcal H$ be a continuous function that satisfies 
$f \in L^{1}(\delta;+\infty ;\mathcal H)$.
Suppose that $\alpha > 1$ and $x : [\delta ;+\infty [\to \mathcal H$ is a classical solution of
$$t\ddot{x}(t)+\alpha\dot{x}(t)=f(t).$$
Then, $x(t)$ converges strongly in $\mathcal H$ as $t \to \infty$.
\end{lemma}
\begin{Theorem}
Suppose that  $\alpha > 3$, $\displaystyle{\int_{t_0}^{+\infty}}tg(t)dt<+\infty$, and  $\Phi$ satisfises $int(\argmin \Phi)\neq \emptyset$. Let $x(\cdot)$ be a classical global solution of equation (\ref{edo01}).
Then, there exists some $x^{*} \in \argmin \,\Phi$ such that $x(t) \to x^{*}$ strongly as $t \to +\infty$.
\end{Theorem}
\begin{proof}
We follow the same approach as that proposed in \cite[Theorem 3.1]{APR1}. We first observe that the assumption $int(\argmin \Phi)\neq \emptyset$ implies the existence of  some $\bar{z}\in \mathcal H$ and  $\rho > 0$ such that, for all $x\in \mathcal H$, $\langle\nabla\Phi (x),x-\bar{z} \rangle\geq\rho\Vert\nabla\Phi(x)\Vert$.
In particular, for all $t \geq t_0$
$$\langle\nabla\Phi (x(t)),x(t)-\bar{z} \rangle\geq\rho\Vert\nabla\Phi(x(t))\Vert .$$
 Combining this inequality with (\ref{rfastoche1}) (that we recall below)
$$\dot{\mathcal E}_{\alpha,g} (t)= \frac{4}{\alpha-1}t (\Phi(x(t))- \inf_{\mathcal H}\Phi ) -  2 t \langle  x(t) - \bar{z} , \nabla \Phi (x(t))   \rangle
$$
we obtain
\begin{equation}\label{Strong1}
\dot{\mathcal E}_{\alpha,g} (t)+2\rho t\Vert\nabla\Phi (x(t))\Vert\leq  \frac{4}{\alpha-1}t (\Phi(x(t))- \inf_{\mathcal H}\Phi ).
\end{equation}
Let us return to
(\ref{basic-energy-Liap}), which after integration,  and using $\alpha >3$, gives 
$$\int_{t_{0}}^{\infty}t(\Phi (x(t))-\inf_{\mathcal H}\Phi ) dt<+\infty.$$
 As a consequence, by integrating  (\ref{Strong1}), we deduce that
$$\int_{t_{0}}^{\infty}t\Vert\nabla\Phi (x(t))\Vert dt<+\infty.$$
By setting $f(t)=tg(t)-t\nabla \Phi (x(t))$, we can rewrite equation (\ref{edo01}) as
$$t\ddot{x}(t)+\alpha\dot{x}(t)=f(t).$$
 Since all assumptions of Lemma \ref{strong-lem1} are satisfied, we can affirm that $x(t)$ converges strongly  to some $x^{*}\in \mathcal H$. Recalling that $\Phi(x(t))\rightarrow\inf_{\mathcal H}\Phi$ and that $\Phi$ is continuous, we obtain  $x^{*}\in \argmin \,\Phi$.
\end{proof}

\smallskip

\noindent \textbf{Strong convergence in the case of an even function.}\\
Recall  that $\Phi:\mathcal H \rightarrow \mathbb R$ is an even function if $\Phi(-x)=\Phi(x)$ for all $x\in \mathcal H$. In this case, $0\in\argmin_{\mathcal{H}}\Phi$.
\begin{Theorem}
Suppose that $\alpha >3$, $\displaystyle{\int_{t_0}^{+\infty}}tg(t)dt<+\infty$,  and  $\Phi$ is an even function. Let $x(\cdot)$ be a classical global solution of equation (\ref{edo01}).
Then, there exists some $\bar{x}\in \argmin_{\mathcal{H}}\Phi$ such that $x(t)$ converges strongly to $\bar{x}$ as $t\to +\infty$. 
\end{Theorem}
\begin{proof}
Set, for $t_0 \leq \tau \leq r$,
$$y(\tau )=\Vert x(\tau)\Vert^{2}-\Vert x(r)\Vert^{2}-\frac{1}{2}\Vert x(\tau ) -x(r)\Vert^{2}.$$
By derivating twice, we obtain
$$\dot{y}(\tau)=\langle \dot{x}(\tau), x(\tau)+x(r) \rangle$$
and
$$\ddot{y}(\tau)=\Vert \dot{x}(\tau)\Vert^{2}+\langle \ddot{x}(\tau), x(\tau)+x(r) \rangle.$$
From these two equations and (1), we deduce that
\begin{equation}\label{even} 
\begin{array}{lll}
\ddot{y}(\tau)+\frac{\alpha}{\tau}\dot{y}(\tau)&=&\Vert \dot{x}(\tau)\Vert^{2}+\langle \ddot{x}(\tau)+\frac{\alpha}{\tau}\dot{x}(\tau), x(\tau)+x(r) \rangle \vspace{2mm}\\
 &=& \Vert \dot{x}(\tau)\Vert^{2}+\langle g(\tau)-\nabla\Phi (x(\tau)), x(\tau)+x(r) \rangle .
\end{array}
\end{equation}
Let us now consider the energy function, $W(\tau )=\frac{1}{2}\Vert \dot{x}(\tau)\Vert^{2}+\Phi(x(\tau))+\int_{\tau}^{\infty}\langle\dot{x}(t),g(t)dt\rangle$. We have $\frac{d}{d\tau}W(\tau)=-\frac{\alpha}{\tau}\Vert \dot{x}(\tau)\Vert^{2}$, and therefore $W$ is a nonincreasing function. As a consequence, $W(\tau)\geq W(r)$, which equivalently gives
$$\frac{1}{2}\Vert \dot{x}(\tau)\Vert^{2}+\Phi(x(\tau))\geq \frac{1}{2}\Vert \dot{x}(r)\Vert^{2}+\Phi(x(r))-\int_{\tau}^{r}\langle\dot{x}(t),g(t)\rangle dt.$$
Using the convex differential inequality $\Phi (-x(r))\geq \Phi (x(\tau))-\langle\nabla\Phi (x(\tau)), x(\tau)+x(r)\rangle$, and the even property of $\Phi$, $\Phi(x(r))=\Phi(-x(r))$, we  deduce that 
$$\frac{1}{2}\Vert \dot{x}(\tau)\Vert^{2}\geq -\langle\nabla\Phi (x(\tau)), x(\tau)+x(r)\rangle -\int_{\tau}^{r}\langle\dot{x}(t),g(t)\rangle dt.$$
Returning to (\ref{even}), we finally obtain
$$\ddot{y}(\tau)+\frac{\alpha}{\tau}\dot{y}(\tau)\leq\frac{3}{2}\Vert \dot{x}(\tau)\Vert^{2}+\langle g(\tau),x(\tau)+x(r)\rangle +\int_{\tau}^{r}\langle\dot{x}(t),g(t)\rangle dt .$$
Let us recall that, by Theorem \ref{Thm-weak-conv}, the trajectory $x(\cdot)$ is converging weakly, and hence bounded. Moreover, by  (\ref{conv1}), we have
$\| \dot{x}(t)  \| \leq \frac{C}{t} $.
Hence, for some constant $C$ 
\begin{equation}\label{even1} 
\ddot{y}(\tau)+\frac{\alpha}{\tau}\dot{y}(\tau)\leq k(\tau):= \frac{3}{2}\Vert \dot{x}(\tau)\Vert^{2}+ C \| g(\tau)\| + C\int_{\tau}^{+ \infty}\frac{1}{t}\|g(t)\| dt. 
\end{equation}
Let us observe that the function $k$ does not depend on $r$. Let us verify that $\tau \mapsto \tau k(\tau) \in L^{1}(t_0,+\infty)$. By Theorem \ref{Thm-weak-conv}, we have 
$\int_{t_0}^{\infty}  t\| \dot{x}(t) \|^2  dt < + \infty$. By assumption, $\int_{t_0}^{+\infty}tg(t)dt<+\infty$.
Moreover, by Fubini theorem
$$
\int_{t_0}^{\infty}\tau \int_{\tau}^{\infty} \frac{1}{t}\|g(t)\| dt d\tau \leq \frac{1}{2} \int_{t_0}^{\infty} t\|g(t)\| dt < +\infty. 
$$
By integration of (\ref{even1}), by a similar argument as in Lemma \ref{basic-edo},  we obtain  
\begin{equation}\label{even2}
\dot{y}(\tau) \leq  \frac{C}{\tau^{\alpha} }  +  \frac{1}{\tau^{\alpha} }  \int_{t_0}^\tau  u^{\alpha} k(u) du,
\end{equation}
where $C={t_0}^{\alpha}\|\dot x(t_0)\|\,\|x\|_{\infty}$. 
Set
$$
K(t):= \frac{C}{\tau^{\alpha} }  +  \frac{1}{\tau^{\alpha} }  \int_{t_0}^\tau  u^{\alpha} k(u) du .
$$
By using Fubini theorem once more, and the fact that 
$\tau \mapsto \tau k(\tau) \in L^{1}(t_0,+\infty)$, we deduce that $K\in  L^{1}(t_0,+\infty)$.
 Integrating $ \dot{y}(\tau) \leq  K(\tau) $ from $t$ to $r$, we obtain
$$
 \frac{1}{2} \|x(t)- x(r) \|^2 \leq  \| x(t) \|^2 -\| x(r) \|^2 
+ \int_t^r K(\tau) d\tau.
$$
Since $\Phi$ is even, we have $0 \in\argmin \Phi$. Hence 
$\lim_{t\to +\infty }\| x(t) \|^2$ exists (see the proof of Theorem \ref{Thm-weak-conv}). As a consequence, $x(t)$ has the Cauchy property as $t \to + \infty$, and hence converges.

\if {
After multiplication of this inequality  by $\tau$, and integration on $[t_0, \theta]$,  with $\theta >r$, we obtain
\begin{equation}\label{even1} 
\int_{t_0}^{\theta}|\tau\kappa(\tau)|dt\leq \frac{3}{2}\int_{t_0}^{+\infty}\tau\Vert\dot{x}(\tau)\Vert^{2}d\tau
+ C \int_{t_0}^{+\infty}\tau \Vert g(\tau)\Vert d\tau +C\int_{t_0}^{\theta}\tau\int_{\tau}^{r}\frac{1}{t}\Vert g(t)\Vert dt d\tau .
\end{equation}
Integrating by parts the last term in this inequality,
\[
\begin{array}{lll}
\int_{t_0}^{\theta}\tau\int_{\tau}^{r}\Vert g(t)\Vert dt d\tau &=& \left[ \frac{\tau^{2}}{2}\int_{\tau}^{r}\frac{1}{t}\Vert g(t)\Vert dt\right]_{\delta}^{\theta}+\int_{t_0}^{\theta}\frac{\tau^{2}}{2}\times \frac{1}{\tau}g(\tau)d\tau \vspace{2mm}\\
&\leq & -\frac{\theta^{2}}{2}\int_{r}^{\theta} \Vert g(t)\Vert dt+\frac{1}{2}\int_{\delta}^{+\infty}\tau^{2}g(\tau)d\tau .
\end{array}
\]
All the terms in the second part on equation (\ref{even1}) are in $L^{1}(\delta,+\infty)$ and we conclude that $\dot{y}^{+}\in L^{1}(\delta,+\infty),$ which implies that the limit of $y(t)$ exists, as $t\to +\infty$.
}
\fi

\end{proof}

\section{The case $\argmin \Phi =\emptyset .$}
\begin{Theorem}\label{Thm-armin-empty}
Suppose $\alpha >0 $, $\displaystyle{\int_{t_0}^{+\infty} \|g(t)\| dt < + \infty}$, and $\inf \Phi >-\infty$.  Then, for any  orbit  $x: [t_0, +\infty[ \rightarrow \mathcal H$   of $ \mbox{{\rm(AVD)}}_{\alpha,g}$,  the following
minimizing property holds
$$\lim_{t\rightarrow +\infty}\Phi (x(t))=\inf_{\mathcal H}\Phi. $$
\end{Theorem}

We will use the following lemma, see \cite{APR1}.
\begin{lemma}\label{basic-int} 
Take $\delta >0$, and let $f \in L^1 (\delta , +\infty)$ be nonnegative. Consider a nondecreasing continuous function $\psi:(\delta,+\infty)\to(0,+\infty)$ such that $\lim\limits_{t\to+\infty}\psi(t)=+\infty$. Then, 
$$\lim_{t \rightarrow + \infty} \frac{1}{\psi(t)} \int_{\delta}^t \psi(s)f(s)ds =0.$$ 
\end{lemma}
\textit{Proof of Theorem }\ref{Thm-armin-empty}.
Let us first return to the proof of the energy estimates in Proposition \ref{energy-thm-1}. 
Replacing $\inf \Phi$ by $\min \Phi$ in the expression of the energy function, we obtain by the same argument
\begin{align}\label{basic-est-empty}
  &  \sup_t \|  \dot{x}(t)  \| < + \infty  ,      \\ 
  & \int_{t_0}^{+\infty} \frac{1}{t} \| \dot{x}(t)  \|^2 dt < + \infty .   
\end{align}
Consider the function 
$h(t)=\frac{1}{2}\Vert x(t)-z\Vert^{2}$, where this time, $z$ is an arbitrary element of $\mathcal H$. We can easily verify that
$$\ddot{h}(t)+\frac{\alpha}{t}\dot{h}(t)= \Vert \dot{x}(t)\Vert^{2}-\langle \nabla \Phi(x(t)),x(t)-z\rangle+\langle g(t),x(t)-z\rangle.$$
By convexity of $\Phi$, we obtain
\begin{equation}\label{h1} 
\ddot{h}(t)+\frac{\alpha}{t}\dot{h}(t)+\Phi (x(t))-\Phi (z)\leq \Vert \dot{x}(t)\Vert^{2}+\langle g(t),x(t)-z\rangle.
\end{equation}
Consider the energy function
$$W(t)=\frac{1}{2}\Vert \dot{x}(t)\Vert^{2}+\Phi (x(t))-\inf \Phi +\int_{t}^{\infty}\langle\dot{x}(s),g(s)\rangle ds.$$
By classical derivation rules, and (\ref{edo01})
\[
\begin{array}{lll}
\displaystyle{\frac{d}{dt}}W(t)&=& \langle \dot{x}(t),\ddot{x}(t)+\nabla\Phi(x(t))-g(t)\rangle\\
    &=& -\dfrac{\alpha}{t}\Vert \dot{x}(t)\Vert^{2}\leq 0.
\end{array}
\]
As a consequence, $W$ is  a nonincreasing function. Moreover, $W$  is minorized by $-\Vert \dot{x}\Vert_{L^{\infty}}\int_{t_0}^{+\infty} \|g(s)\|ds.$ Hence, there exists some $W_{\infty} \in \mathbb R$ such that $W(t)\rightarrow W_{\infty}$ as $t\rightarrow\infty$.
Let us take advantage of this property, and reformulate  
(\ref{h1}) with the help of $W$:
 \begin{equation}\label{h1b} 
\ddot{h}(t)+\frac{\alpha}{t}\dot{h}(t)+W(t) + \inf \Phi -\Phi (z)\leq \frac{3}{2}\Vert \dot{x}(t)\Vert^{2}+\langle g(t),x(t)-z\rangle  +\int_{t}^{\infty}\langle\dot{x}(s),g(s)\rangle ds .
 \end{equation}
For every $t>0$, \ $ W(t)\geq W_{\infty}$. Setting $B_{\infty}=W_{\infty} +\inf \Phi-\Phi(z) $, we obtain
$$B_{\infty}\leq \frac{3}{2}\Vert \dot{x}(t)\Vert^{2}+\Vert g(t)\Vert\Vert x(t)-z\Vert+\Vert \dot{x}\Vert_{L^{\infty}}
\int_{t}^{\infty}\Vert g(s)\Vert ds-\frac{1}{t^{\alpha}}\frac{d}{dt}(t^{\alpha}\dot{h}(t)).$$
 Multiplying this last equation by $\frac{1}{t}$, and integrating between two reals $0< t_{0}<\theta$, we get
\begin{equation}\label{h2} 
B_{\infty} \ln (\frac{\theta}{t_{0}})\leq \frac{3}{2} \int_{t_{0}}^{\theta}\frac{1}{t}\Vert \dot{x}(t)\Vert^{2}dt+\int_{t_{0}}^{\theta}
\dfrac{\Vert g(t)\Vert\Vert x(t)-z\Vert}{t} dt+\Vert \dot{x}\Vert_{L^{\infty}} \int_{t_{0}}^{\theta} \left( \frac{1}{t}\int_{t}^{\infty}\Vert g(s)\Vert ds\right)  dt
-\int_{t_{0}}^{\theta}\frac{1}{t^{\alpha+1}}\frac{d}{dt}(t^{\alpha}\dot{h}(t))dt.
\end{equation}
Let us estimate the integrals in the second member of (\ref{h2}):
\begin{enumerate}
\item By (\ref{basic-est-empty}), $\int_{t_0}^{+\infty} \frac{1}{t} \| \dot{x}(t)  \|^2 dt < + \infty $.
\item
Exploiting the relation $\Vert x(t)-z\Vert \leq \Vert x(t_0)-z\Vert +\int_{t_0}^{t}\Vert\dot{x}(s)\Vert ds$, we obtain
$$\int_{t_{0}}^{\theta}\dfrac{\Vert g(t)\Vert\Vert x(t)-z\Vert}{t} dt\leq \left(  \frac{\Vert x_{0}-z\Vert}{t_{0}}+
\Vert \dot{x}\Vert_{L^{\infty}}\right) \int_{t_{0}}^{+\infty}\Vert g(t)\Vert dt< +\infty.$$
\item After integration by parts
$$
\int_{t_{0}}^{\theta} \left( \frac{1}{t}\int_{t}^{\infty}\Vert g(s)\Vert ds\right)  dt = \ln \theta \int_{\theta}^{\infty}\Vert g(s)\Vert ds - \ln t_0 \int_{t_0}^{\infty}\Vert g(s)\Vert ds + \int_{t_{0}}^{\theta} \Vert g(t)\Vert\ln t  \ dt.
$$
\item Set $I=\int_{t_{0}}^{\theta}\frac{1}{t^{\alpha+1}}\frac{d}{dt}(t^{\alpha}\dot{h}(t))dt.$ By integrating by parts twice
\[
\begin{array}{lll}
I&=& \left[ \frac{1}{t}\dot{h}(t)\right]_{t_{0}}^{\theta} +(\alpha+1)\int_{t_{0}}^{\theta}\frac{1}{t^{2}}\dot{h}(t)dt\\
 &=& C +\frac{1}{\theta}\dot{h}(\theta)  +\frac{(1+\alpha )}{\theta^{2}}h(\theta )+2(1+\alpha)\int_{t_{0}}^{\theta}
\frac{1}{t^{3}}h(t)dt.

\end{array}
\] 
Since $h\geq 0$, we have $-I\leq -C -\frac{1}{\theta}\dot{h}(\theta).$  Then notice  that 
$\vert\dot{h}(\theta )\vert=\vert\langle \dot{x}(\theta),x(\theta )-z\rangle\vert\leq \Vert\dot{x}\Vert_{L^{\infty}}(\Vert x(0)-z\Vert+\theta\Vert\dot{x}\Vert_{L^{\infty}})$. 
\end{enumerate}
Collecting the above results, we deduce from (\ref{h2}) that
\begin{equation}\label{h2b} 
B_{\infty} \ln (\frac{\theta}{t_{0}})\leq C +
\ln \theta \int_{\theta}^{\infty}\Vert g(s)\Vert ds  +\Vert \dot{x}\Vert_{L^{\infty}} \int_{t_{0}}^{\theta} \Vert g(t)\Vert\ln t  \ dt.
\end{equation}
 Dividing by $\ln (\frac{\theta}{t_{0}})$, and letting $\theta\rightarrow +\infty$,  thanks to Lemma \ref{basic-int} with $\psi (t) =\ln t$, we conclude that $B_{\infty} \leq 0.$ Equivalently, for every $z\in \mathcal H$, $W_{\infty}\leq \Phi (z)-\inf \Phi$, which leads to $W_{\infty} \leq 0.$ 
 
 On the other hand,  it is easy to see that $W(t)\geq \Phi(x(t))-\inf \Phi -\Vert\dot{x}\Vert_{L^{\infty}}\int_{t}^{+\infty}g(s)ds$. 
Passing to the limit, as $t\rightarrow+\infty$, we deduce that
$$0\geq W_{\infty}\geq\limsup \Phi(x(t))- \inf \Phi .$$ Since we always have $\inf \Phi\leq \liminf \Phi(x(t))$, we conclude that $\lim_{t\rightarrow +\infty}\Phi(x(t))=\inf\Phi.$
\quad \quad  \quad  \quad  \quad  \quad $\square$

\begin{remark} In \cite{APR}, in the unperturbed case $g=0$, it has been observed that, when $\argmin \Phi =\emptyset$, the fast convergence property of the values, as given in Theorem \ref{fastconv-thm}, may fail to be satisfied. A fortiori, without making additional assumption on the perturbation term, we also loose the fast convergence property in the perturbed case (take $g=0$!).
\end{remark}

\section{From continuous to discrete dynamics and algorithms}

Time discretization of  dissipative gradient-based dynamical systems leads naturally  to algorithms, which, under appropriate assumptions, have similar convergence properties.
This  approach has been followed successfully in a variety of situations. For a general abstract discussion see \cite {Alv_Pey2}, \cite{Alv_Pey3}, and  in the case or dynamics with inertial features see \cite{Al}, \cite{AA1}, \cite{aabr}, \cite{APR}, \cite{APR1}.
To cover  practical situations involving constraints and/or nonsmooth data, we  need to broaden our scope. 
This leads us to consider the non-smooth structured convex minimization problem
\begin{equation}\label{algo1}
\min \left\lbrace  \Phi (x) + \Psi (x): \ x \in \mathcal H   \right\rbrace   
\end{equation}
where

$\bullet$ $\Phi: \mathcal H \to  \mathbb R \cup \lbrace   + \infty  \rbrace  $ is a convex lower semicontinuous proper function  (which possibly takes the value $+ \infty$);
\vspace{1mm}

$\bullet$ $\Psi: \mathcal H \to  \mathbb R  $ is a convex continuously differentiable function, whose gradient is Lipschitz continuous. 

\medskip

\noindent The optimal solutions of (\ref{algo1}) satisfy
$$
\partial \Phi (x) + \nabla \Psi (x) \ni 0,
$$
where $\partial \Phi$ is the subdifferential of $\Phi$ in the sense of convex analysis.
In order to adapt our dynamic to this non-smooth situation, we will consider the corresponding differential inclusion
\begin{equation}\label{algo2}
\ddot{x}(t) + \frac{\alpha}{t} \dot{x}(t)  + \partial \Phi (x(t)) + \nabla \Psi (x(t)) \ni g(t).
\end{equation}
This dynamic is within the following framework
\begin{equation}\label{algo2b}
\ddot{x}(t) + a(t) \dot{x}(t)  + \partial \Theta (x(t)) \ni g(t),
\end{equation}
where $\Theta: \mathcal H \to  \mathbb R \cup \lbrace   + \infty  \rbrace  $  is a convex lower semicontinuous proper function,  and $a(\cdot)$ is a positive damping parameter.

The detailed study of this  differential inclusion goes far beyond the scope of the present article, see \cite{ACR} for some results in the case of a fixed positive damping parameter, i.e., $a(t)= \gamma >0$ fixed, and  $g=0$.
A formal analysis of this sytem shows that the Lyapunov analysis, which has been developed in the previous sections, still holds, as long as one does not use the Lipschitz continuity property of the gradient (cocoercivity property). This is based on the fact that the convexity (subdifferential) inequalites are still valid, as well as the (generalized) derivation chain rule, see \cite{Bre1}.
Thus,
setting $\Theta (x)= \Phi (x) + \Psi (x)$,  we can reasonably assume that, for $\alpha >3$, and $\int_{t_0}^{+\infty} t \|g(t)\| dt < + \infty$, for each trajectory of  (\ref{algo2}), there is rapid convergence  of the values,
\begin{align*}
  \Theta(x(t))-  \min\Theta \leq \frac{C}{t^2} ,
\end{align*}
 and weak convergence of the trajectory to an optimal solution.

Indeed, we are  going to use these ideas  as a guideline, and so introduce corresponding fast converging algorithms, making the link with Nesterov \cite{Nest1}-\cite{Nest4}, Beck-Teboulle \cite{BT}, and so extending the recent works of Chambolle-Dossal \cite{CD},  Su-Boyd-Cand\`es
\cite{SBC}, Attouch-Peypouquet-Redont \cite{APR} to the perturbed case.
As a basic ingredient of the discretization procedure, in order to preserve the fast convergence properties of the dynamical system (\ref{algo2}), we  are going to discretize it \textit{implicitely} with respect to the nonsmooth function $\Phi$, and \textit{explicitely} with respect to the smooth function $\Psi$.

Taking a fixed time step size $h>0$, and setting $t_k= kh$, $x_k = x(t_k)$ the implicit/explicit finite difference scheme for (\ref{algo2}) gives
\begin{equation}\label{algo3}
\frac{1}{h^2}(x_{k+1} -2 x_{k}   + x_{k-1} ) +\frac{\alpha}{kh^2}( x_{k}  - x_{k-1})  + \partial \Phi (x_{k+1}  ) + \nabla \Psi (y_k) \ni g_k,
\end{equation}
where  $y_k$ is a linear combination of $x_k$ and $x_{k-1}$, that will be made precise further.
After developing (\ref{algo3}), we obtain 
\begin{equation}\label{algo4}
x_{k+1} + h^2 \partial \Phi (x_{k+1}) \ni  x_{k} + \left( 1- \frac{\alpha}{k}\right) ( x_{k}  - x_{k-1}) - h^2 \nabla \Psi (y_k) + h^2 g_k .
\end{equation}
A natural choice for $y_k$ leading to a simple formulation of the algorithm (other choices are possible, offering new directions of research  for the future) is
\begin{equation}\label{algo5}
y_k=   x_{k} + \left( 1- \frac{\alpha}{k}\right) ( x_{k}  - x_{k-1}).
\end{equation}
Using the classical proximal operator (equivalently, the  resolvent operator of the maximal monotone operator $\partial \Phi$)
\begin{equation}\label{algo6}
\mbox{prox}_{ \gamma \Phi } (x)= {\argmin}_{\xi \in  \mathcal H} \left\lbrace \Phi (\xi) + \frac{1}{2 \gamma} \| \xi -x  \|^2
\right\rbrace  = \left(I + \gamma \partial \Phi \right)^{-1} (x)
\end{equation}
and setting $s=h^2$, the algorithm can be written as
\begin{equation}\label{algo7}
 \left\{
\begin{array}{l}
y_k=   x_{k} + \left( 1- \frac{\alpha}{k}\right) ( x_{k}  - x_{k-1});  	   \\
\rule{0pt}{20pt}
 x_{k+1} = \mbox{prox}_{s \Phi} \left( y_k- s (\nabla \Psi (y_k) -g_k) \right). 
 \end{array}\right.
\end{equation}
For practical purpose, and in order to fit with the existing litterature on the subject, it is convenient to work with the following equivalent formulation
\begin{equation}\label{algo7b}
 {\rm \mbox{(AVD)}_{\alpha,g}-algo}  \ \left\{
\begin{array}{l}
y_k=   x_{k} + \frac{k -1}{k  + \alpha -1} ( x_{k}  - x_{k-1});  	   \\
\rule{0pt}{20pt}
 x_{k+1} = \mbox{prox}_{s \Phi} \left( y_k- s (\nabla \Psi (y_k) -g_k)\right). 
 \end{array}\right.
\end{equation}
Indeed, we have $\frac{k -1}{k  + \alpha -1} = 1- \frac{\alpha}{k  + \alpha -1}$. When $\alpha$ is an integer, up to the reindexation 
$k\mapsto k  + \alpha -1$, we obtain the same sequences $(x_k)$ and $(y_k)$. For general $\alpha >0$, we can easily verify that the algorithm
$ {\rm \mbox{(AVD)}_{\alpha,g}-algo} $ is still associated with the dynamical system  (\ref{algo2}).

This algorithm is within the scope of the  proximal-based inertial algorithms \cite{AA1}, \cite{MO}, \cite{LP}, and forward-backward methods. In the unperturbed case, $g_k =0$,
it has been recently considered by Chambolle-Dossal \cite{CD},  Su-Boyd-Cand\`es \cite{SBC}, and Attouch-Peypouquet-Redont \cite{APR}. It enjoys fast convergence properties which are very similar to that of the continuous dynamic. 

 For $\alpha = 3$, $g_k =0$, we recover the classical algorithm based on Nesterov and G\"uler ideas, and developed by Beck-Teboulle (FISTA)
\begin{equation}\label{algo7c}
 \ \left\{
\begin{array}{l}
y_k=   x_{k} + \frac{k -1}{k  + 2} ( x_{k}  - x_{k-1});  	   \\
\rule{0pt}{20pt}
 x_{k+1} = \mbox{prox}_{s \Phi} \left( y_k- s \nabla \Psi (y_k) \right). 
 \end{array}\right.
\end{equation}
%

An important question regarding the (FISTA) method, as described in (\ref{algo7c}), is the convergence of sequences $(x_k)$ and $(y_k)$. Indeed, it is still an open question.
A major interest to consider the  broader context of $ {\rm \mbox{(AVD)}_{\alpha,g}-algo}$ algorithms is that, for $\alpha >3$, these sequences  converge, and they allow errors/perturbations, and using approximation methods. We will see that the proof of the convergence properties of $ {\rm \mbox{(AVD)}_{\alpha,g}-algo}$ algorithms  can be obtained  in a parallel  way with the convergence analysis in the continuous case in Theorem \ref{Thm-weak-conv}.

\begin{Theorem} \label{Thm-algo} \
Let  $\Phi: \mathcal H \to  \mathbb R \cup \lbrace   + \infty  \rbrace  $ be a convex lower semicontinuous proper function, and $\Psi: \mathcal H \to  \mathbb R  $  a convex continuously differentiable function, whose gradient is $L$-Lipschitz continuous.
Suppose  that   $S=\argmin (\Phi + \Psi)$ is nonempty.
Suppose that $\alpha \geq 3$, \ $ 0< s < \frac{1}{L} $, and \ $ \sum_{k \in \mathbb N}  k \|g_k\| < + \infty  $. 
Let $(x_k)$ be a sequence generated by  the algorithm {\rm${\rm \mbox{(AVD)}_{\alpha,g}-algo} $}.
Then,
 $$(\Phi + \Psi)(x_k)-  \min(\Phi + \Psi) =\mathcal O (\frac{1}{k^2}).$$
  Precisely,
\begin{equation}\label{algo8}
  (\Phi + \Psi)(x_k)-  \min(\Phi + \Psi) \leq  \frac{C (\alpha -1)}{2s\left( k + \alpha -2\right)^2 },
\end{equation}
with $C$  given by
$$
C =  \mathcal G (0) + 2s\left( \sum_{j=0}^{\infty} \left( j + \alpha -1\right)    \| g_j\|\right)  \left( \sqrt{\frac{\mathcal E (0)}{\alpha -1}} + \frac{2s}{\alpha -1} \sum_{j=0}^{\infty} \left( j + \alpha -1\right)   \| g_j\| \right) ,
$$
where
$$
\mathcal G (0) = \frac{2s}{\alpha -1} \left(  \alpha -2\right)^2    (\Theta (x_0) - {\Theta}^*) + (\alpha -1)
 \| y_0 - x^{*} \|^2 .
$$
%
%
\end{Theorem}

\begin{proof} 
To simplify notations, we set $\Theta = \Phi + \Psi$, and take $x^{*} \in \argmin \Theta$, i.e., $\Theta (x^{*})= \inf \Theta  $.
In a parallel way  to the continuous case, our proof is based on proving that $(\mathcal E (k))$ is a non-increasing sequence, where $\mathcal E (k)$  is the discrete version of 
the Lyapunov function ${\mathcal E}_{\alpha,g} (t)$ (we shall justify further that it is well defined), and which is given by
\begin{equation}\label{algo9b}
\mathcal E (k):= \frac{2s}{\alpha -1} \left( k + \alpha -2\right)^2    (\Theta (x_k) - \Theta(x^{*}) + (\alpha -1)
 \| z_k -x^{*} \|^2  + \sum_{j=k}^{\infty} 2s\left( j + \alpha -1\right)    \left\langle g_j, z_{j+1}- x^{*}  \right\rangle ,
\end{equation}
with
\begin{equation}\label{algo9c}
z_k := \frac{k + \alpha -1}{\alpha -1}y_k - \frac{k}{\alpha -1}x_k .
\end{equation}
In the passage from the continuous to discrete, we recall that we must use the reindexing
$k \mapsto k + \alpha -1$. Note that $\mathcal E (k)$ is equal to the Lyapunov function considered by Su-Boyd-Cand\`es in 
\cite[Theorem 4.3]{SBC}, plus a perturbation term.\\
 Let us introduce the function $\Psi_k : \mathcal H \rightarrow \mathbb R $ which is defined by
$$
\forall y \in \mathcal H, \ \Psi_k (y):= \Psi (y) - \left\langle g_k, y \right\rangle .
$$
We also set
$$
\Theta_k = \Phi + \Psi_k .
$$
We have  $\nabla \Psi_k (y)= \nabla\Psi (y) -g_k$, and hence $\nabla \Psi_k$ is still $L$-Lipschitz continuous.
We can reformulate our algorithm with the help of $\Psi_k $ as follows
\begin{equation}\label{algo12}
 {\rm \mbox{(AVD)}_{\alpha,g}-algo}  \ \left\{
\begin{array}{l}
y_k=   x_{k} + \frac{k -1}{k  + \alpha -1} ( x_{k}  - x_{k-1});  	   \\
\rule{0pt}{20pt}
 x_{k+1} = \mbox{prox}_{s \Phi} \left( y_k- s \nabla \Psi_k (y_k)\right). 
 \end{array}\right.
\end{equation}
In order to analyze the convergence properties of the above algorithm, it is convenient to introduce the 
operator $G_{s,k}: \mathcal H \rightarrow \mathcal H$ which is defined by, for all $y \in \mathcal H$,
$$
G_{s,k} (y) = \frac{1}{s}\left( y - \mbox{prox}_{s \Phi} \left( y- s \nabla \Psi_k (y)  \right) \right) .
$$
Equivalently,
$$
\mbox{prox}_{s \Phi} \left( y- s \nabla \Psi_k (y)  \right)  = y - s G_{s,k} (y),
$$
and the  algorithm (\ref{algo12}) can be formulated as
\begin{equation}\label{algo13}
 {\rm \mbox{(AVD)}_{\alpha,g}-algo}  \ \left\{
\begin{array}{l}
y_k=   x_{k} + \frac{k -1}{k  + \alpha -1} ( x_{k}  - x_{k-1});  	   \\
\rule{0pt}{20pt}
 x_{k+1} = y_k - s G_{s,k} (y_k).
  \end{array}\right.
\end{equation}
The variable $z_k$, which is defined in (\ref{algo9c}) by \ $z_k = \frac{k + \alpha -1}{\alpha -1}y_k - \frac{k}{\alpha -1}x_k$,
will play an important role. It comes naturally into play as a discrete version of the term 
$  \frac{t}{\alpha-1}  \dot{x}(t) + x(t) - x^{*} $
which enters  ${\mathcal E}_{\alpha,g}(t)$. Indeed,
\begin{align}\label{algo13-b}
\frac{k + \alpha -1}{\alpha -1} \left(  x_{k+1}- x_{k} \right) + x_{k}& = \frac{k + \alpha -1}{\alpha -1} x_{k+1}
-\frac{k }{\alpha -1} x_k \\
 & = z_{k+1} \nonumber
\end{align}
where the last equality comes from
(\ref{algo14b}) below.
Let us examine the recursive relation satisfied by $z_k$.
We have
\begin{align}
z_{k+1}&= \frac{k + \alpha}{\alpha -1}y_{k+1} - \frac{k+1}{\alpha -1}x_{k+1} \nonumber \\
&= \frac{k + \alpha}{\alpha -1}\left(  x_{k+1} + \frac{k }{k  + \alpha } ( x_{k+1}  - x_{k})            \right)  - \frac{k+1}{\alpha -1}x_{k+1} \nonumber\\
&= \frac{k + \alpha -1}{\alpha -1}  x_{k+1}         
-\frac{k}{\alpha -1}x_{k} \label{algo14b}\\
&= \frac{k + \alpha -1}{\alpha -1}   \left(y_k - s G_{s,k} (y_k) \right)  -\frac{k}{\alpha -1}x_{k} \nonumber\\
&= z_k -\frac{s}{\alpha -1} \left( k + \alpha -1\right)  G_{s,k} (y_k) \label{algo14a}.
\end{align}
We now use the classical formula in the proximal gradient 
(also called forward-backward) analysis (see \cite{BT}, \cite{CD}, \cite{PB}, \cite{SBC}): for any $x, y\in \mathcal H$
\begin{equation}\label{algo14}
 \Theta_k (y - sG_{s,k} (y)) \leq \Theta_k (x) + \left\langle  G_{s,k} (y), y-x \right\rangle -\frac{s}{2} \|  G_{s,k} (y) \|^2 .
\end{equation}
Note that this formula is valid since $s \leq \frac{1}{L}$, and $\nabla \Psi_k$ is $L$-lipschitz continuous.
Let us write successively this formula at $y=y_k$ and $x= x_k$, then at 
$y=y_k$ and $x= x^{*}$. We obtain 
\begin{align*}
 &\Theta_k (y_k - sG_{s,k} (y_k)) \leq \Theta_k (x_k) + \left\langle  G_{s,k} (y_k), y_k-x_k \right\rangle -\frac{s}{2} \|  G_{s,k} (y_k) \|^2 \\
& \Theta_k (y_k - sG_{s,k} (y_k)) \leq \Theta_k (x^{*}) + \left\langle  G_{s,k} (y_k), y_k-x^{*} \right\rangle -\frac{s}{2} \|  G_{s,k} (y_k) \|^2 .
\end{align*}
Multiplying the first equation by $\frac{k}{k + \alpha -1}$, and the second by $\frac{\alpha -1}{k + \alpha -1}$, then adding the two resulting equations, and using $x_{k+1} = y_k - s G_{s,k} (y_k)$, we obtain
\begin{align}\label{algo15}
 \Theta_k (x_{k+1}) \leq &\frac{k}{k + \alpha -1} \Theta_k (x_k) + \frac{\alpha -1}{k + \alpha -1}\Theta_k (x^{*}) 
 -\frac{s}{2} \|  G_{s,k} (y_k) \|^2 \\
 & + \left\langle  G_{s,k} (y_k), \frac{k}{k + \alpha -1}(y_k-x_k ) + \frac{\alpha -1}{k + \alpha -1} (y_k-x^{*})\right\rangle \label{algo15b}.
\end{align}
Let us rewrite the  scalar product in  (\ref{algo15b})    as follows:
\begin{align}\label{algo16}
 \left\langle  G_{s,k} (y_k), \frac{k}{k + \alpha -1}(y_k-x_k ) + \frac{\alpha -1}{k + \alpha -1} (y_k-x^{*})\right\rangle &=  \frac{\alpha -1}{k + \alpha -1} 
 \left\langle  G_{s,k} (y_k), \frac{k}{\alpha -1}(y_k-x_k ) + y_k-x^{*}\right\rangle   \\
& =\frac{\alpha -1}{k + \alpha -1} 
 \left\langle  G_{s,k} (y_k), \frac{k+\alpha -1}{\alpha -1}y_k- \frac{k}{\alpha -1} x_k -x^{*}\right\rangle \nonumber\\
& =\frac{\alpha -1}{k + \alpha -1} 
 \left\langle  G_{s,k} (y_k), z_k -x^{*}\right\rangle  \nonumber.
\end{align}
Combining (\ref{algo15})-(\ref{algo15b})     with (\ref{algo16}), we obtain
\begin{align}\label{algo17}
 \Theta_k (x_{k+1}) \leq \frac{k}{k + \alpha -1} \Theta_k (x_k) + \frac{\alpha -1}{k + \alpha -1}\Theta_k (x^{*}) +
 \frac{\alpha -1}{k + \alpha -1} 
 \left\langle  G_{s,k} (y_k), z_k -x^{*}\right\rangle -\frac{s}{2} \|  G_{s,k} (y_k) \|^2 .
\end{align}
In order to write (\ref{algo17}) in a recursive form, we use the  relation (\ref{algo14a}) satisfied by $z_k$, which gives
$$
z_{k+1} -x^{*} = z_k  -x^{*} -\frac{s}{\alpha -1} \left( k + \alpha -1\right)  G_{s,k} (y_k) .
$$   
After developing
$$
\| z_{k+1} -x^{*} \|^2 = \| z_{k} -x^{*} \|^2 
-2\frac{s}{\alpha -1} \left( k + \alpha -1\right) 
\left\langle   z_{k} -x^{*}, G_{s,k} (y_k)                     \right\rangle  + \frac{s^2}{(\alpha -1)^2} \left( k + \alpha -1\right)^2  \|  G_{s,k} (y_k)  \|^2 ,
$$
and multiplying the above expression by 
$\frac{(\alpha -1)^2} {2s\left( k + \alpha -1\right)^2}$,
we obtain
$$
\frac{(\alpha -1)^2} {2s\left( k + \alpha -1\right)^2}
\left( \| z_{k} -x^{*} \|^2 -\| z_{k+1} -x^{*} \|^2 \right)  = \frac{\alpha -1}{k + \alpha -1} 
 \left\langle  G_{s,k} (y_k), z_k -x^{*}\right\rangle -\frac{s}{2} \|  G_{s,k} (y_k) \|^2 .
$$
Replacing this expression in (\ref{algo17}), we obtain
\begin{align}\label{algo18}
 \Theta_k (x_{k+1}) \leq \frac{k}{k + \alpha -1} \Theta_k (x_k) + \frac{\alpha -1}{k + \alpha -1}\Theta_k (x^{*}) +
 \frac{(\alpha -1)^2} {2s\left( k + \alpha -1\right)^2}
\left( \| z_{k} -x^{*} \|^2 -\| z_{k+1} -x^{*} \|^2 \right) .
\end{align}
Equivalently
\begin{align}\label{algo19}
 \Theta_k (x_{k+1})-\Theta_k (x^{*}) \leq \frac{k}{k + \alpha -1} \left( \Theta_k (x_k) -\Theta_k (x^{*}) \right) + \frac{(\alpha -1)^2} {2s\left( k + \alpha -1\right)^2}
\left( \| z_{k} -x^{*} \|^2 -\| z_{k+1} -x^{*} \|^2 \right) .
\end{align}
Returning to $\Theta (y)= \Theta_k (y) + \left\langle g_k, y \right\rangle$, we obtain
\begin{align}\label{algo19b}
 \Theta(x_{k+1})-\Theta (x^{*}) \leq &\frac{k}{k + \alpha -1} \left( \Theta (x_k) -\Theta(x^{*}) \right) + \frac{(\alpha -1)^2} {2s\left( k + \alpha -1\right)^2}
\left( \| z_{k} -x^{*} \|^2 -\| z_{k+1} -x^{*} \|^2 \right) \\
& +       \left\langle g_k, x_{k+1}- x^{*} \right\rangle
 - \frac{k}{k + \alpha -1} \left\langle g_k, x_{k}- x^{*} \right\rangle    \nonumber.
\end{align}
After reduction
\begin{align}\label{algo20}
 \Theta(x_{k+1})-\Theta (x^{*}) \leq &\frac{k}{k + \alpha -1} \left( \Theta (x_k) -\Theta(x^{*}) \right) + \frac{(\alpha -1)^2} {2s\left( k + \alpha -1\right)^2}
\left( \| z_{k} -x^{*} \|^2 -\| z_{k+1} -x^{*} \|^2 \right) \\&+       \left\langle g_k, x_{k+1}- x_{k} 
+ \frac{\alpha -1}{k + \alpha -1}(x_{k}- x^{*}  )\right\rangle
    \nonumber.
\end{align}
Multiplying by $ \frac{2s}{\alpha -1}\left( k + \alpha -1\right)^2    $, we obtain
\begin{align}\label{algo21}
\frac{2s}{\alpha -1}\left( k + \alpha -1\right)^2    \left( \Theta(x_{k+1})-\Theta (x^{*})\right)  \leq &\frac{2s}{\alpha -1} k\left( k + \alpha -1\right) \left( \Theta (x_k) -\Theta(x^{*}) \right) + (\alpha -1)
\left( \| z_{k} -x^{*} \|^2 -\| z_{k+1} -x^{*} \|^2 \right) \\&+   \frac{2s}{\alpha -1}\left( k + \alpha -1\right)^2       \left\langle g_k, x_{k+1}- x_{k} 
+ \frac{\alpha -1}{k + \alpha -1}(x_{k}- x^{*}  )\right\rangle 
    \nonumber.
\end{align}
For $\alpha \geq 3$ one can easily verify that 
$$
k\left( k + \alpha -1\right) \leq \left( k + \alpha -2\right)^2   .
$$
More precisely
$$
k\left( k + \alpha -1\right) =\left( k + \alpha -2\right)^2 -k(\alpha -3) -(\alpha -2)^2 \leq \left( k + \alpha -2\right)^2 -k(\alpha -3).
$$
As a consequence, from (\ref{algo21})   we deduce that 
\begin{align}\label{algo22}
\frac{2s}{\alpha -1}&\left( k + \alpha -1\right)^2    \left( \Theta(x_{k+1})-\Theta (x^{*})\right) +  2s \frac{\alpha -3}{\alpha -1}k \left( \Theta (x_k) -\Theta(x^{*}) \right)  \leq \frac{2s}{\alpha -1} \left( k + \alpha -2\right)^2 \left( \Theta (x_k) -\Theta(x^{*}) \right) \\& +(\alpha -1)
\left( \| z_{k} -x^{*} \|^2 -\| z_{k+1} -x^{*} \|^2 \right) +   \frac{2s}{\alpha -1}\left( k + \alpha -1\right)^2       \left\langle g_k, x_{k+1}- x_{k} 
+ \frac{\alpha -1}{k + \alpha -1}(x_{k}- x^{*}  )\right\rangle 
    \nonumber.
\end{align}
Setting
\begin{equation}\label{algo23}
   \mathcal G(k)= \frac{2s}{\alpha -1} \left( k + \alpha -2\right)^2    (\Theta (x_k) - {\Theta}^*) + (\alpha -1)
 \| z_{k} -x^{*} \|^2  ,
\end{equation}
we can reformulate (\ref{algo22}) as
\begin{equation}\label{algo28}
   \mathcal G (k+1) +  2s \frac{\alpha -3}{\alpha -1}k \left( \Theta (x_k) -\Theta(x^{*}) \right) \leq  \mathcal G (k) + \frac{2s}{\alpha -1}\left( k + \alpha -1\right)^2    \left\langle g_k, x_{k+1}- x_{k} 
+ \frac{\alpha -1}{k + \alpha -1}(x_{k}- x^{*}  )\right\rangle.
\end{equation}
Equivalently
\begin{equation}\label{algo29}
   \mathcal G (k+1) +  2s \frac{\alpha -3}{\alpha -1}k \left( \Theta (x_k) -\Theta(x^{*}) \right) \leq  \mathcal G (k) + 2s\left( k + \alpha -1\right)    \left\langle g_k, \frac{k + \alpha -1}{\alpha -1} \left(  x_{k+1}- x_{k} \right) + x_{k}- x^{*}  \right\rangle.
\end{equation}
Using (\ref{algo13-b})
\begin{align*}
 z_{k+1} = \frac{k + \alpha -1}{\alpha -1} \left(  x_{k+1}- x_{k} \right) + x_{k},
\end{align*}
we deduce that
\begin{equation}\label{algo30}
   \mathcal G (k+1) +  2s \frac{\alpha -3}{\alpha -1}k \left( \Theta (x_k) -\Theta(x^{*}) \right) \leq  \mathcal G (k) + 2s\left( k + \alpha -1\right)    \left\langle g_k, z_{k+1}- x^{*}  \right\rangle.
\end{equation}
We now develop a similar analysis as in the continuous case. Given some integer $K$, set 
$$
\mathcal E_K (k)= \mathcal G (k) + \sum_{j=k}^K 2s\left( j + \alpha -1\right)    \left\langle g_j, z_{j+1}- x^{*}  \right\rangle .
$$
Then (\ref{algo30}) is equivalent to
\begin{equation}\label{algo31}
   \mathcal E_K (k+1) +  2s \frac{\alpha -3}{\alpha -1}k \left( \Theta (x_k) -\Theta(x^{*}) \right) \leq  \mathcal E_K (k) .
\end{equation}
Hence, the sequence $(\mathcal E_K (k))$ is nonincreasing.
In particular $\mathcal E_K (k) \leq \mathcal E_K (0)$, which gives
$$
\mathcal G (k) + \sum_{j=k}^K 2s\left( j + \alpha -1\right)    \left\langle g_j, z_{j+1}- x^{*}  \right\rangle \leq \mathcal G (0) + \sum_{j=0}^K 2s\left( j + \alpha -1\right)    \left\langle g_j, z_{j+1}- x^{*}  \right\rangle .
$$
As a consequence
\begin{equation}\label{algo31b}
\mathcal G (k) \leq \mathcal G (0) + \sum_{j=0}^{k-1} 2s\left( j + \alpha -1\right)    \left\langle g_j, z_{j+1}- x^{*}  \right\rangle .
\end{equation}
By definition of $\mathcal G (k)$, neglecting some positive terms, and by Cauchy-Schwarz inequality, we infer
$$
(\alpha -1)
 \| z_{k} -x^{*} \|^2  \leq \mathcal G (0) + 2s \sum_{j=0}^{k-1} \left( j + \alpha -1\right)   \| g_j\| \|z_{j+1}- x^{*} \| .
$$
Equivalently
\begin{equation}\label{algo32}
 \| z_{k} -x^{*} \|^2  \leq \frac{1}{\alpha -1}\mathcal G (0) + \frac{2s}{\alpha -1} \sum_{j=1}^{k} \left( j + \alpha -2\right)   \| g_{j-1}\| \|z_{j}- x^{*} \| .
\end{equation}
We then use the following result, a discrete version of Gronwall's lemma.
\begin{lemma}\label{d-Gronwall} Let $(a_k)$ be a sequence of positive real numbers such that
$$
a_k^2 \leq c + \sum_{j=1}^k \beta_j  a_j
$$
where $(\beta_j )$ is a sequence of positive real numbers such that $\sum_j \beta_j <+ \infty$, and $c$ is a positive real number. Then
$$
a_k \leq \sqrt{c} + \sum_{j=1}^{\infty} \beta_j .
$$
\end{lemma}
\begin{proof}
Set $A_k := \sup_{1\leq j \leq  k} a_j $.
Then, for $1\leq l \leq k$
$$
a_l^2 \leq c + \sum_{j=1}^l \beta_j  a_j \leq c + 
A_k \sum_{j=1}^{\infty} \beta_j  
$$
Passing to the supremum with respect to $l$, with $1\leq l \leq k$, we obtain
$$
A_k^2 \leq c + 
A_k \sum_{j=1}^{\infty} \beta_j  .
$$
By elementary algebraic computation, it follows that
$$
A_k \leq \sqrt{c} + \sum_{j=1}^{\infty} \beta_j .
$$
\end{proof}
\textit{Following the proof of Theorem \ref{Thm-algo}.}
From (\ref{algo32}), applying Lemma \ref{d-Gronwall} with 
$a_k = \| z_{k} -x^{*} \|$, we deduce that
\begin{equation}\label{algo33}
 \| z_{k} -x^{*} \|  \leq M:= \sqrt{\frac{\mathcal G (0)}{\alpha -1}} + \frac{2s}{\alpha -1} \sum_{j=0}^{\infty} \left( j + \alpha -1\right)   \| g_{j}\|  .
\end{equation}
Note that $M$ is finite, because of the assumption 
$ \sum_{k \in \mathbb N}  k \|g_k \| < + \infty  $.
Returning to (\ref{algo31b}) we obtain
\begin{equation}\label{algo34}
\mathcal G (k) \leq C:= \mathcal G (0) + 2s\left( \sum_{j=0}^{\infty} \left( j + \alpha -1\right)    \| g_j\|\right)  \left( \sqrt{\frac{\mathcal G (0)}{\alpha -1}} + \frac{2s}{\alpha -1} \sum_{j=0}^{\infty} \left( j + \alpha -1\right)   \| g_j\| \right)  .
\end{equation}
By definition of $\mathcal G (k)$, and the positivity of its constitutive elements we finally obtain
$$
\frac{2s}{\alpha -1} \left( k + \alpha -2\right)^2    (\Theta (x_k) - {\Theta}^*) \leq C .
$$
which gives (\ref{algo8}).
\end{proof}

\begin{remark} In  the particular case $\alpha = 3$,  for a perturbed version of the classical FISTA algorithm,   Schmidt, Le Roux, and  Bach proved in \cite{SLB} a  result similar to Theorem \ref{Thm-algo} concerning the fast convergence of the values. 
\end{remark}

Let us now study the convergence of the sequence $(x_k)$.
\begin{Theorem} \label{Thm-algo2} \
Let  $\Phi: \mathcal H \to  \mathbb R \cup \lbrace   + \infty  \rbrace  $ be a convex lower semicontinuous proper function, and $\Psi: \mathcal H \to  \mathbb R  $  a convex continuously differentiable function, whose gradient is $L$-Lipschitz continuous.
Suppose  that   $S=\argmin (\Phi + \Psi)$ is nonempty.
Suppose that $\alpha >3$, \ $ 0< s < \frac{1}{L} $, and \ $ \sum_{k \in \mathbb N}  k \|g_k\| < + \infty  $. 
Let $(x_k)$ be a sequence generated by  the algorithm {\rm${\rm \mbox{(AVD)}_{\alpha,g}-algo} $}.
Then,

\medskip
 i) \ $ \sum_k    k \Big((\Phi + \Psi)(x_k) - \inf(\Phi + \Psi) \Big)   < + \infty$;
 
 \medskip
 
 ii) \ $\sum k\|x_{k+1}-x_k\|^2<+\infty$ ;
 
 \medskip
 
  iii)\ $ (x_k ) \ \mbox{converges weakly, as} \ k\to+\infty, \ \mbox{to some} \ x^*\in\argmin\Phi$.

\end{Theorem}

\begin{proof} 
The demonstration is parallel to that of Theorem \ref{Thm-weak-conv}.

\medskip
\textbf{Step 1.} \ 
Let us return to (\ref{algo30}),
\begin{equation*}
   \mathcal G (k+1) +  2s \frac{\alpha -3}{\alpha -1}k \left( \Theta (x_k) -\Theta(x^{*}) \right) \leq  \mathcal G (k) + 2s\left( k + \alpha -1\right)    \left\langle g_k, z_{k+1}- x^{*}  \right\rangle. 
\end{equation*}
By  (\ref{algo33}), we know that the sequence $ (z_{k})$ is bounded.
Summing the above inequalities, and using  $\alpha >3$, we obtain
\begin{equation} \label{algo-conv1}
\sum_k    k \left((\Phi + \Psi)(x_k) - \inf(\Phi + \Psi) \right)   < + \infty ,
\end{equation}
thats' item $i)$. 

\medskip

\textbf{Step 2.} 
Now apply the fundamental inequality (\ref{algo14}), which  can be equivalently written as follows
\begin{equation}\label{algo-conv2}
 \Theta_k (y -s G_{s,k} (y)) + \frac{1}{2s}\| y- sG_{s,k} (y)-x \|^2      \leq \Theta_k (x) + \frac{1}{2s} \|  x-y \|^2 .
\end{equation}
Take $y= y_k$, and $x=x_k$. Since $x_{k+1} = y_k - s G_{s,k} (y_k)$, and $y_k -   x_{k} = \frac{k -1}{k  + \alpha -1} ( x_{k}  - x_{k-1}) $, we  obtain
\begin{equation}\label{algo-conv3}
 \Theta_k (x_{k+1}) + \frac{1}{2s}\| x_{k+1}-x_k \|^2      \leq \Theta_k (x_k) + \frac{1}{2s}\frac{(k -1)^2}{(k  + \alpha -1)^2}  \|  x_{k}  - x_{k-1} \|^2 .
\end{equation}
Equivalently, by definition of $ \Theta_k $,
\begin{equation}\label{algo-conv4}
 \Theta (x_{k+1}) + \frac{1}{2s}\| x_{k+1}-x_k \|^2      \leq \Theta (x_k) + \frac{1}{2s}\frac{(k -1)^2}{(k  + \alpha -1)^2}  \|  x_{k}  - x_{k-1} \|^2 + \left\langle 
  g_k ,  x_{k+1}-x_k                  \right\rangle .
\end{equation}
To shorten notations, set $\theta_k = \Theta (x_k) -\Theta (x^{*}) $, 
$d_k = \frac{1}{2}\|  x_{k}  - x_{k-1} \|^2 $, $a=\alpha -1$.
By Cauchy-Schwarz inequality, and with these notations,
(\ref{algo-conv4}) gives
\begin{equation}\label{algo-conv5}
\frac{1}{s}\left( d_{k+1} - \frac{(k -1)^2}{(k + a)^2} d_k \right)  \leq \left( \theta_k - \theta_{k+1} \right) +
 \| g_k\| \|  x_{k+1}-x_k \|    .
\end{equation}
After multiplication by $(k + a)^2$, we obtain
\begin{equation}\label{algo-conv6}
\frac{1}{s}\left( (k + a)^2 d_{k+1} - (k -1)^2 d_k \right)  \leq (k + a)^2\left( \theta_k - \theta_{k+1} \right) +
(k + a)^2 \| g_k\| \|  x_{k+1}-x_k \|    .
\end{equation}
Summing from $k=1$ to $k= K$ gives
\begin{equation}\label{algo-conv7}
\sum_{k=1}^{K}\left( (k + a)^2 d_{k+1} - (k -1)^2 d_k \right)  \leq s\sum_{k=1}^{K}(k + a)^2\left( \theta_k - \theta_{k+1} \right) +s\sum_{k=1}^{K}
(k + a)^2 \| g_k\| \|  x_{k+1}-x_k \|    .
\end{equation}
By a similar computation as in Chambolle-Dossal \cite[Corollary 2]{CD}, we equivalently obtain
\begin{align}\label{algo-conv8}
(K+a)^2  d_{K+1}  + &    \sum_{k=2}^{K}a\left( 2k +a -2\right)d_k  \\
&\leq s \left( (a+1)^2 \theta_1 - (K+a)^2 \theta_{K+1} +\sum_{k=2}^{K}\left( 2k +2a -1 \right)\theta_k +\sum_{k=1}^{K}
(k + a)^2 \| g_k\| \|  x_{k+1}-x_k \|  \right)   \nonumber.
\end{align}
By (\ref{algo-conv1}) we have 
$\sum_k    \left( 2k +2a -1 \right)\theta_k < + \infty $.
Hence there exists some constant $C$ such that, for all
$ K \in \mathbb N$
\begin{equation}\label{algo-conv9}
(K+a)^2  \|  x_{K+1}-x_K \|^2 
\leq C +  2s\sum_{k=1}^{K}
(k + a)^2 \| g_k\| \|  x_{k+1}-x_k \|    .
\end{equation}
We now proceed to a parallel argument to that used in the proof of Theorem \ref{Thm-weak-conv}. Let us write (\ref{algo-conv9}) as follows, with $r_k:= (k + a)\|  x_{k+1}-x_k \| $
\begin{equation}\label{algo-conv10}
r_{k}^2 
\leq C +  2s\sum_{j=1}^{k}
(j + a)\| g_j\| r_j   .
\end{equation}
We make appeal to the following discrete version of the Gronwall-Bellman lemma.
\begin{lemma}\label{Gronw-dis} Let $(r_k)$ be sequence of positive real numbers such that, for all $k\geq 1$
$$
r_{k}^2 
\leq C +  \sum_{j=1}^{k}
\omega_j r_j
$$
where $C$ is a positive constant, and $\sum_k \omega_j <+\infty$, with $\omega_j \geq 0$.
 Then the sequence $(r_k)$  is bounded with
$$r_{k} \leq \sqrt{C} + \sum_{j \in \mathbb N} \omega_j  .$$
 \end{lemma}
\begin{proof}
For simplicity, let us assume $\omega_j >0$ (one can always reduce to this situation by adding some positive constant, arbitrarily small, see Brezis \cite{Bre1} for the proof of this lemma in the continuous case).
Set $A_k := C + \sum_{j=1}^{k}
\omega_j r_j$, $A_0 = C$. We have 
$r_{k}^2 \leq A_k $, and $A_{k+1} -A_k =  \omega_{k+1}  r_{k+1}$. Equivalently $   r_{k+1} = \frac{A_{k+1} -A_k}{\omega_{k+1}}$, which gives 
$$
\frac{A_{k+1} -A_k}{\omega_{k+1}} \leq \sqrt{A_{k+1}},
$$
and hence
$$
\frac{A_{k+1}}{\sqrt{A_{k+1}}} - \frac{A_{k}}{\sqrt{A_{k+1}}}\leq \omega_{k+1}.
$$
From this, and using that the sequence $(A_{k})$ is increasing, we deduce that
$$
\sqrt{A_{k+1}} -  \sqrt{A_{k}} \leq \omega_{k+1}.
$$
Summing this inequality, and using $r_{k} \leq \sqrt{A_{k}}  $ gives the claim.
\end{proof}

\textit{Following the proof of Theorem \ref{Thm-algo2}.}
Let us apply lemma \ref{Gronw-dis} to inequality (\ref{algo-conv10}) with  $r_j = (j + a)\|  x_{j+1}-x_j \|$, and $\omega_j  = (j + a)\| g_j\| $. By using the  assumption on the perturbation term $\sum_k  k \|g_k\| <+\infty$, we deduce that 
\begin{equation}\label{algo-conv12}
\sup_k  k \|  x_{k+1}-x_k \|  < + \infty   .
\end{equation}
Injecting this information in (\ref{algo-conv8}), we obtain
\begin{equation}\label{algo-conv13}
 \sum_{k}a\left( 2k +a -2\right)d_k  
\leq C +   \sum_{k}\left( 2k +2a -1 \right)\theta_k +\sup_k ( (k+a) \|  x_{k+1}-x_k \| ) \sum_{k}
(k + a) \| g_k \|     .
\end{equation}
From $a = \alpha -1 \geq 2$, (\ref{algo-conv1}), and the definition of $d_k$, we deduce that
$$ \sum k\|x_{k+1}-x_k\|^2<+\infty,$$ 
which is our claim $ii)$.

\medskip

\textbf{Step 3.} The last step consists in applying Opial's lemma, whose discrete version is stated below.

\begin{lemma}\label{Opial-discrete} Let $S$ be a non empty subset of $\mathcal H$,
and $(x_k)$ a sequence of elements of $\mathcal H$. Assume that 
\begin{eqnarray*}
(i) &  & \mbox{for every }z\in S,\>\lim_{k\to+\infty}\|x_k-z\|\mbox{ exists};\\
(ii) &  & 
\mbox{every weak sequential cluster point of the sequence} \ (x_k) \mbox{ belongs to }S.
\end{eqnarray*}
 Then 
\[
w-\lim_{k\to+\infty}x_k=x_{\infty}\ \ \mbox{ exists, for some element }x_{\infty}\in S.
\]
 \end{lemma}
We are going to apply Opial's lemma with $S=\argmin (\Phi + \Psi)$. 
By Theorem \ref{Thm-algo}, we have $(\Phi + \Psi)(x_k) \to 
\min (\Phi + \Psi)$ (indeed, we have proved fast convergence). By the lower semicontinuity property of $\Phi + \Psi$ for the weak convergence of $\mathcal H$, we immediately obtain that item $(ii)$ of Opial's lemma is satisfied.
Thus the only point  to verify is that $\lim \|x_k-x^{*}\|$ exists for any $x^{*} \in \argmin (\Phi + \Psi)$. Equivalently, we are going to show that $\lim h_k$ exists, with $h_k := \frac{1}{2} \|x_k  - x^{*}\|^2$.

The beginning of the proof is similar to \cite{AA1}, \cite{CD}. It consists in establishing a discrete version of the second-order differential inequality
(\ref{wconv50})
\begin{equation*}
 \ddot{h}(t) + \frac{\alpha}{t} \dot{h}(t) \leq \| \dot{x}(t) \|^2  +       \| x(t) - x^{*} \|  \| g(t) \|.
\end{equation*}
We use the  parallelogram identity, which in an equivalent form can be written as follows: for any  $a,b,c \in \mathcal H$
\begin{equation}\label{algo-paral}
\frac{1}{2} \|a-b  \|^2 + \frac{1}{2} \|a-c  \|^2 =
\frac{1}{2} \|b -c \|^2  + \left\langle a-b, a-c \right\rangle .
\end{equation} 
Taking $b= x^{*}$, $a= x_{k+1}$, $c=x_k$, we obtain
$$
\frac{1}{2} \|x_{k+1}-x^{*}  \|^2 + \frac{1}{2} \|x_{k+1}-x_k  \|^2 =
\frac{1}{2} \|  x_k - x^{*}\|^2  + \left\langle x_{k+1}-x^{*}, x_{k+1}-x_k\right\rangle .
$$ 
Equivalently,
\begin{equation}\label{algo-conv14}
h_k - h_{k+1} = \frac{1}{2} \|x_{k+1}-x_k  \|^2  + \left\langle x_{k+1}-x^{*}, x_k - x_{k+1}\right\rangle .
\end{equation}
By definition of $y_k$ we have 
$$
x_k - x_{k+1}= y_k - x_{k+1} - \frac{k -1}{k + \alpha -1}
(x_k - x_{k-1}).
$$
Replacing in (\ref{algo-conv14}), we obtain
\begin{equation}\label{algo-conv15}
h_k - h_{k+1} = \frac{1}{2} \|x_{k+1}-x_k  \|^2  + \left\langle x_{k+1}-x^{*},y_k - x_{k+1}\right\rangle -
\frac{k -1}{k + \alpha -1}\left\langle x_{k+1}-x^{*},x_k - x_{k-1}\right\rangle .
\end{equation}
Let us now use the monotonicity property of $\partial \Phi$. Since $- s\nabla \Psi (x^{*}) \in  s\partial \Phi ( x^{*})    $, and $y_k - x_{k+1}- s \nabla \Psi (y_k) +sg_k\in  s\partial \Phi (x_{k+1})$, we have
$$
\left\langle y_k - x_{k+1}- s \nabla \Psi (y_k) +sg_k +s\nabla \Psi (x^{*}) , x_{k+1} -  x^{*}   \right\rangle \geq 0.
$$
Equivalently
$$
\left\langle y_k - x_{k+1}, x_{k+1} -  x^{*} \right\rangle + s \left\langle \nabla \Psi (x^{*}) - \nabla \Psi (y_k) +g_k   , x_{k+1} -  x^{*}   \right\rangle \geq 0.
$$
Replacing in (\ref{algo-conv15}) we obtain
\begin{equation}\label{algo-conv16}
 h_{k+1}- h_k  +\frac{1}{2} \|x_{k+1}-x_k  \|^2  + s \left\langle  \nabla \Psi (y_k) -\nabla \Psi (x^{*}) - g_k   , x_{k+1} -  x^{*}   \right\rangle -
\frac{k -1}{k + \alpha -1}\left\langle x_{k+1}-x^{*},x_k - x_{k-1}\right\rangle \leq 0.
\end{equation}
We now use the co-coercivity of $\nabla \Psi$
\begin{align}\label{algo-conv17}
\left\langle  \nabla \Psi (y_k) -\nabla \Psi (x^{*})  , x_{k+1} -  x^{*}   \right\rangle  &= \left\langle  \nabla \Psi (y_k) -\nabla \Psi (x^{*})    , x_{k+1} -  y_k   \right\rangle + \left\langle  \nabla \Psi (y_k) -\nabla \Psi (x^{*})    ,   y_k  -x^{*} \right\rangle
\nonumber\\
&\geq \frac{1}{L} \|  \Psi (y_k) -\nabla \Psi (x^{*})  \| ^2 +  \left\langle  \nabla \Psi (y_k) -\nabla \Psi (x^{*})    , x_{k+1} -  y_k   \right\rangle \nonumber \\
&\geq \frac{1}{L} \|  \Psi (y_k) -\nabla \Psi (x^{*})  \| ^2 -  \|  \nabla \Psi (y_k) -\nabla \Psi (x^{*})  \|  \|x_{k+1} -  y_k    \|\\
&\geq -\frac{L}{2} \| x_{k+1} -  y_k \| ^2 
\nonumber.
\end{align}
Combining (\ref{algo-conv16}) and (\ref{algo-conv17})
\begin{equation}\label{algo-conv18}
 h_{k+1}- h_k  +\frac{1}{2} \|x_{k+1}-x_k  \|^2  -\frac{sL}{2} \| x_{k+1} -  y_k \| ^2 - s\|g_k \| \|x_{k+1} -  x^{*}\| -
\frac{k -1}{k + \alpha -1}\left\langle x_{k+1}-x^{*},x_k - x_{k-1}\right\rangle \leq 0.
\end{equation}
Let us use again (\ref{algo-paral}) with $b= x^{*}$, $a= x_{k}$, $c=x_{k -1}$. We obtain
\begin{equation*}
\frac{1}{2} \|x_{k}-x^{*}  \|^2 + \frac{1}{2} \|x_{k}-x_{k -1} \|^2 =
\frac{1}{2} \|x_{k -1} -x^{*} \|^2  + \left\langle x_{k}-x^{*}, x_{k}-x_{k -1} \right\rangle .
\end{equation*} 
Equivalently
\begin{equation}\label{algo-paral-3}
 h_{k-1}-h_{k}  = \frac{1}{2} \|x_{k}-x_{k -1}  \|^2
- \left\langle x_{k}-x^{*}, x_{k}-x_{k -1} \right\rangle .
\end{equation} 
Combining (\ref{algo-conv18}) with (\ref{algo-paral-3}) we obtain
\begin{align}\label{algo-conv19}
 h_{k+1}- h_k  - &\frac{k -1}{k + \alpha -1}\left(  h_{k} -h_{k-1}\right) \leq
 -\frac{1}{2} \|x_{k+1}-x_k  \|^2  +\frac{sL}{2} \| x_{k+1} -  y_k \| ^2 + s\|g_k \| \|x_{k+1} -  x^{*}\|\\
 & +\frac{k -1}{k + \alpha -1}\left(\frac{1}{2} \|x_k -x_{k -1}  \|^2 + \left\langle x_k - x_{k-1}, x_{k+1}-x_k\right\rangle            \right) .\nonumber
\end{align}
By definition of $y_k =     x_{k} + \frac{k -1}{k  + \alpha -1} ( x_{k}  - x_{k-1})$, we have 
$x_{k+1}- y_k =   x_{k+1} - x_{k} - \frac{k -1}{k  + \alpha -1} ( x_{k}  - x_{k-1})$. Hence
$$
\|x_{k+1}- y_k \|^2=  \| x_{k+1} - x_{k}\|^2 + \left( \frac{k -1}{k  + \alpha -1} \right) ^2 \|x_{k}  - x_{k-1}\|^2 -2\frac{k -1}{k  + \alpha -1} \left\langle x_{k+1} - x_{k}, x_{k}  - x_{k-1}\right\rangle 
$$
Substituting in (\ref{algo-conv19}), we obtain
\begin{equation}\label{algo-conv20}
 h_{k+1}- h_k  - \gamma_k \left(  h_{k} -h_{k-1}\right) \leq 
 -(1-  \frac{sL}{2}) \| x_{k+1} -  y_k \| ^2 + s\|g_k \| \|x_{k+1} -  x^{*}\|   + \left( \gamma_k + {\gamma_k}^2 \right)\|x_{k}  - x_{k-1}\|^2 ,
\end{equation}
where $\gamma_k= \frac{k -1}{k  + \alpha -1} $.
Since $0< s < \frac{1}{L} $, we have $(1-  \frac{sL}{2}) >0$. On the other hand, since $\gamma_k <1$, we have 
$\gamma_k + {\gamma_k}^2 < 2 \gamma_k$. Hence
\begin{equation}\label{algo-conv21}
 h_{k+1}- h_k  - \gamma_k \left(  h_{k} -h_{k-1}\right) \leq 
  s\|g_k \| \|x_{k+1} -  x^{*}\| + 2 \gamma_k  \|x_{k}  - x_{k-1}\|^2 .
\end{equation}
By (\ref{algo33}), we know that the sequence $(z_k)$ is bounded. By (\ref{algo-conv12}), we know that  $
\sup_k  k \|  x_{k+1}-x_k \|  < + \infty $  .
Since $x_k = z_k - \frac{k + \alpha -1}{ \alpha -1}  ( x_{k+1}-x_k )$, we deduce that the sequence $(x_k)$ is bounded.
Returning to (\ref{algo-conv21}), we have, for some constant $C$
\begin{equation}\label{algo-conv22}
 h_{k+1}- h_k  - \gamma_k \left(  h_{k} -h_{k-1}\right) \leq 
  C\|g_k \| + 2 \gamma_k  \|x_{k}  - x_{k-1}\|^2 .
\end{equation}
We now use the estimation that we  obtained in step 2, namely $ \sum_k k\|x_{k+1}-x_k\|^2<+\infty$.
Combined with the assumption 
$ \sum_{k }  k \|g_k\| < + \infty  $, we deduce that
\begin{equation}\label{algo-conv23}
 h_{k+1}- h_k  - \gamma_k \left(  h_{k} -h_{k-1}\right) \leq  \omega_k ,
\end{equation}
for some nonnegative sequence $(\omega_k)  $ such that 
$  \sum_{k \in \mathbb N} k\omega_k < + \infty $.
Taking the positive part, we obtain
\begin{equation}\label{algo-conv24}
\left(  h_{k+1}- h_k\right)^+  - \gamma_k \left(  h_{k} -h_{k-1}\right)^+ \leq  \omega_k .
\end{equation}
We are now using the following lemma, which is a discrete version of lemma \ref{basic-edo}.
\begin{lemma}\label{diff-ineq-disc} Let $(a_k)$ be sequence of nonnegative real numbers such that, for all $k\geq 1$
$$
a_{k+1} 
\leq \frac{k -1}{k  + \alpha -1}a_k + 
\omega_k
$$
where $\alpha \geq 3$, and $\sum_k k\omega_k <+\infty$, with $\omega_k \geq 0$.
 Then the sequence $(a_k)$  is summable, i.e., 
$$ \sum_{k \in \mathbb N} a_k < +\infty  .$$
 \end{lemma}
\begin{proof}
Since $\alpha \geq 3$ we have $\alpha -1\geq 2$, and hence
$$
a_{k+1} 
\leq \frac{k -1}{k  + 2}a_k + \omega_k .
$$
Multiplying this expression by $(k+1)^2$, we obtain
$$
(k+1)^2 a_{k+1} 
\leq \frac{(k -1)(k+1)^2}{k  + 2}a_k + (k+1)^2\omega_k .
$$
Then note that, for all integer $k$
$$
\frac{(k -1)(k+1)^2}{k  + 2} \leq k^2 .
$$
Hence
$$
(k+1)^2 a_{k+1} 
\leq k^2 a_k + (k+1)^2\omega_k .
$$
Summing this inequality with respect to $j=1,2,...,k$, we obtain
$$
k^2 a_{k} 
\leq  a_1 + \sum_{j=1}^{k-1}(j+1)^2\omega_j .
$$
Dividing by $k^2$, and summing with respect to $k$, we obtain
$$
\sum _k a_{k} 
\leq  a_1 \sum_k \frac{1}{k^2}    + \sum_k \frac{1}{k^2}\sum_{j=1}^{k-1}(j+1)^2\omega_j .
$$
Applying Fubini theorem to this last sum, we obtain
$$
\sum_k a_{k} 
\leq  a_1 \sum_k \frac{1}{k^2}    + \sum_j \left( \sum_{k=j+1}^{\infty}\frac{1}{k^2} \right) (j+1)^2\omega_j .
$$
We have
$$
\sum_{k=j+1}^{\infty}\frac{1}{k^2} \leq \int_{j}^{\infty}\frac{1}{t^2}dt = \frac{1}{j}.
$$
Hence
$$
\sum_k a_{k} 
\leq  a_1 \sum \frac{1}{k^2}    + \sum_j \frac{(j+1)^2}{j} \omega_j <+\infty,
$$
which by $\frac{(j+1)^2}{j} \leq 4j$ for $j\geq 1$ gives the claim.
\end{proof}
\textit{End of the proof of Theorem \ref{Thm-algo2}.}
Let us apply lemma \ref{diff-ineq-disc} with 
$a_k = \left(  h_{k} -h_{k-1}\right)^+$.    We obtain
$$\sum_{k }\left(  h_{k} -h_{k-1}\right)^+ < +\infty ,$$
which, combined with $h_k$ nonnegative,  gives the convergence of the sequence $(h_k)$, and ends the proof.
\end{proof}


\smallskip

\begin{thebibliography}{10}

\bibitem{AAS} {\sc B. Abbas, H. Attouch, B. F. Svaiter}, {\em  Newton-like dynamics
and forward-backward methods for structured monotone inclusions in Hilbert spaces}, JOTA, DOI 10.1007/s10957-013-0414-5, (2013).

\medskip

 \bibitem{AAC}{\sc S. Adly, H. Attouch, A. Cabot}, {\em Finite time stabilization of nonlinear oscillators subject to dry friction},
Nonsmooth Mechanics and Analysis (edited by P. Alart, O. Maisonneuve and R.T. Rockafellar),
Adv. in Math. and Mech., Kluwer (2006), pp.~289--304.

\medskip

\bibitem{Al}{\sc F. Alvarez},  {\em On the minimizing property of a second-order dissipative system in Hilbert spaces},
SIAM J. Control Optim., 38, No 4, (2000), pp. 1102-1119. 

\medskip

\bibitem{AA1} {\sc F. Alvarez, H. Attouch}, {\em An inertial proximal method for maximal monotone operators via discretization of a nonlinear oscillator with damping},
     Set-Valued Analysis,  9 (2001), No. 1-2, pp.  3--11.
     
\medskip

 \bibitem{AA} {\sc F. Alvarez, H. Attouch}, {\em Convergence and asymptotic stabilization for some damped hyperbolic equations with non-isolated equilibria},
ESAIM Control Optim. Calc. of Var.,  6 (2001), pp.  539--552.

\medskip

\bibitem{aabr}{\sc F. Alvarez, H. Attouch, J. Bolte, P. Redont}, {\em A second-order gradient-like dissipative dynamical system with Hessian-driven damping. Application to optimization and mechanics},   J. Math. Pures Appl.,  81 (2002), No. 8, pp.  747--779.
    
\medskip


\bibitem{Alv_Pey2} {\sc F. Alvarez, J. Peypouquet}, {\em Asymptotic almost-equivalence of Lipschitz evolution systems in Banach spaces},  Nonlinear Anal.,  73 (2010), No. 9, pp. 3018--3033.

\medskip

\bibitem{Alv_Pey3} {\sc F. Alvarez, J. Peypouquet}, {\em A unified approach to the asymptotic almost-equivalence of evolution systems without Lipschitz conditions}, Nonlinear Anal. 74 (2011), No. 11, pp. 3440--3444.

\medskip

\bibitem{ABM}{\sc H. Attouch, G. Buttazzo, G. Michaille}, {\em Variational analysis in Sobolev and BV spaces. Applications to PDE's and optimization},  MPS/SIAM Series on Optimization, 6, Society for Industrial and Applied Mathematics (SIAM), Philadelphia, PA, Second edition, 2014, 793 pages.
 
 \medskip
  
 \bibitem{ACR} {\sc H. Attouch, A.  Cabot, P. Redont}, {\em The dynamics of elastic shocks via epigraphical
 regularization of a differential inclusion},   Adv. Math. Sci. Appl.,  12   (2002), No.1,  pp.  273--306.
 
  \medskip

\bibitem{AtCz1}{\sc H. Attouch, M.-O. Czarnecki}, {\em Asymptotic control and stabilization 
of nonlinear oscillators with non-isolated equilibria}, J. Differential Equations, 179 (2002), pp.~278--310.
 
\medskip


\bibitem{AGR}{\sc H. Attouch, X.  Goudou and P. Redont},  {\em The heavy ball with friction method. The continuous dynamical system, global exploration of the local minima
 of a real-valued function by asymptotical analysis of a dissipative dynamical system},
 Commun. Contemp. Math., 2  (2000), No 1, pp.~1--34.  

\medskip

\bibitem{AMR} {\sc H. Attouch, P.E. Maing\'e, P. Redont},  {\em A second-order differential system with Hessian-driven damping; Application to non-elastic shock laws},
  Differential Equations and Applications,  4 (2012), No. 1, pp. 27--65.

\medskip


\bibitem{APR} {\sc H. Attouch, J. Peypouquet, P. Redont},  {\em A dynamical approach to an inertial forward-backward algorithm for convex minimization},
SIAM J. Optim., 24  (2014), No. 1, pp.~232--256. 

\medskip

\bibitem{APR1} {\sc H. Attouch, J. Peypouquet, P. Redont},  {\em Fast convergence of an inertial dynamical system with asymptotic vanishing damping}, 2015,
to appear.

\medskip

\bibitem{AS}{\sc H. Attouch, A. Soubeyran}, {\em Inertia and reactivity in 
decision making as cognitive variational inequalities}, Journal of Convex 
Analysis, 13 (2006), pp.~207-224.

\medskip

\bibitem{Ba}{\sc J.-B. Baillon}, {\em Un exemple concernant le comportement asymptotique de la solution du probl\`eme $\frac{du}{dt} + \partial \phi(u) \ni 0$}, Journal of Functional Analysis, 28 (1978), pp.~369-376.

\medskip

\bibitem{BC}{\sc H. Bauschke, P. Combettes}, {\em Convex Analysis and Monotone Operator Theory in Hilbert spaces }, CMS Books in Mathematics, Springer,   (2011).


\medskip

\bibitem{BT}{\sc A. Beck, M. Teboulle},  {\em A fast iterative shrinkage-thresholding algorithm for linear inverse
problems},  SIAM J. Imaging Sci., 2(1) 2009, pp.~183--202.

\medskip

\bibitem{Bre1}{\sc H. Br\'ezis}, {\em Op\'erateurs maximaux monotones dans les 
espaces de Hilbert et \'equations d'\'evolution}, Lecture Notes 5, North Holland, (1972).

\medskip

\bibitem{Bre2}{\sc H. Br\'ezis}, {\em Asymptotic behavior of some evolution systems: Nonlinear evolution equations},
 Academic Press, New York, (1978), pp.~141--154.
 
\medskip 

\bibitem{Bruck} {\sc R.E. Bruck},  {\em  Asymptotic convergence of nonlinear contraction
semigroups in Hilbert spaces}, J. Funct. Anal.,  \textbf{18 } (1975),
pp.~15--26. 

 
 \medskip
 \bibitem{Cabot-inertiel}{\sc A. Cabot}, {\em Inertial gradient-like dynamical system controlled by a stabilizing term}, J. Optim. Theory Appl., 120 (2004), pp.~275--303.
 

\medskip


 \bibitem{CEG1}{\sc  A. Cabot, H. Engler, S. Gadat}, {\em On the long time behavior of second order differential equations
with asymptotically small dissipation}
Transactions of the American Mathematical Society, 361 (2009), pp.~5983--6017.

\medskip

\bibitem{CEG2}{\sc  A. Cabot, H. Engler, S. Gadat}, {\em Second order differential equations with asymptotically small dissipation
and piecewise flat potentials},
Electronic Journal of Differential Equations, 17 (2009), pp.~33--38. 

\medskip


\bibitem{CD}{\sc  A. Chambolle, Ch. Dossal}, {\em On the convergence of the iterates of Fista},
HAL Id: hal-01060130
https://hal.inria.fr/hal-01060130v3
Submitted on 20 Oct 2014.

\medskip

\bibitem{CPS}{\sc R. Cominetti; J. Peypouquet, S. Sorin},  {\em Strong asymptotic convergence of evolution equations governed by maximal monotone operators with Tikhonov regularization}, Journal of Differential Equations 245 (2008), No 12, pp. 3753--3763.

\medskip

\bibitem{FB}{\sc N. Flammarion, F. Bach}, {\em From averaging to acceleration, there is only a step-size},
(2015) HAL-01136945.

\medskip

\bibitem{JM}{\sc M.A. Jendoubi, R. May}, {\em
Asymptotics for a second-order differential equation with
nonautonomous damping and an integrable source term}, Applicable Analysis, 2014.

\medskip

\bibitem{MO}{\sc A. Moudafi, M. Oliny}, {\em Convergence of a splitting inertial proximal method for monotone operators}, J. Comput. Appl. Math., 155 (2), (2003),  pp.   447--454.

\medskip

\bibitem{Nest1}{\sc  Y. Nesterov}, {\em  A method of solving a convex programming problem with convergence rate
O(1/k2)}. In Soviet Mathematics Doklady, volume 27,  1983, pp.~ 372--376.

\medskip

\bibitem{Nest2}{\sc  Y. Nesterov}, {\em Introductory lectures on convex optimization: A basic course}, volume 87 of
Applied Optimization. Kluwer Academic Publishers, Boston, MA, 2004.

\medskip

\bibitem{Nest3}{\sc  Y. Nesterov}, {\em Smooth minimization of non-smooth functions}, Mathematical programming,
103(1) 2005, pp.~127--152.

\medskip

\bibitem{Nest4}{\sc  Y. Nesterov}, {\em Gradient methods for minimizing composite objective function}, CORE Discussion
Papers, 2007.

\medskip


\bibitem{Op} {\sc Z. Opial}, {\em Weak convergence of the sequence of successive approximations for nonexpansive mappings},  Bull. Amer. Math.
Soc.,  \textbf{73}  (1967), pp. 591--597.

\medskip

\bibitem{DC}{\sc  B. O’Donoghue, E. J. Cand`es}, {\em Adaptive restart for accelerated gradient schemes}, Found.
Comput. Math., 2013.

\medskip

\bibitem{PB}{\sc N.  Parikh,  S. Boyd}, {\em 
Proximal algorithms},  Foundations and trends in optimization, volume 1, (2013), pp. 123-231.
preprint.

\medskip

\bibitem{Pey}{\sc J. Peypouquet}, {\em Analyse asymptotique de syst\`emes d'\'evolution et applications en optimisation}, Th\`ese Universit\'e Paris 6, (2008).

\medskip


\bibitem{LP}{\sc D. A. Lorenz, Thomas Pock}, {\em An inertial forward-backward algorithm
for monotone inclusions}, J. Math. Imaging Vision, pages 1-15,
2014. (online).

\medskip

\bibitem{SLB}{\sc
M. Schmidt, N. Le Roux, F. Bach}, {\em  Convergence rates of inexact proximal-gradient methods for convex optimization}, NIPS'11 - 25 th Annual Conference on Neural
Information Processing Systems, Dec 2011, Grenada, Spain. (2011) HAL inria-00618152v3.

\medskip

\bibitem{SBC}{\sc W.  Su,  S. Boyd,  E. J. Cand\`es}, {\em 
A Differential Equation for Modeling Nesterov's
Accelerated Gradient Method: Theory and Insights}.
preprint.






\end{thebibliography}
\end{document}